\documentclass[11pt]{article}
\usepackage{amsmath, amssymb, amscd, amsthm, amsfonts, mathtools}
\usepackage{graphicx}
\usepackage{hyperref}
\usepackage{indentfirst}
\usepackage{comment}
\usepackage[dvipsnames]{xcolor}

\newcommand{\SL}{\mathrm{SL}}

\newcommand{\Prob}{\mathrm{Prob}}

\newcommand{\bord}{\mathrm{Bord }}
\newcommand{\overbar}[1]{\mkern 1.5mu\overline{\mkern-1.5mu#1\mkern-1.5mu}\mkern 1.5mu}

\DeclarePairedDelimiter\ceil{\lceil}{\rceil}
\DeclarePairedDelimiter\floor{\lfloor}{\rfloor}

\definecolor{uuuuuu}{rgb}{0.26666666666666666,0.26666666666666666,0.26666666666666666}
\allowdisplaybreaks

\theoremstyle{plain}
\newtheorem{theorem}{Theorem} 
\newtheorem{proposition}{Proposition}
\newtheorem{corollary}{Corollary}
\newtheorem{lemma}{Lemma}

\theoremstyle{definition}
\newtheorem{definition}{Definition} 

\newtheorem{example}{Example} 
\newtheorem*{remark}{Remark}

\title{Continuity of the drift in groups acting on strongly hyperbolic spaces}
\author {Luís Miguel Sampaio}

\begin{document}

	\maketitle
	
	\begin{abstract}
		The Avalanche principle, in its original setting, together with large deviations yields a systematic way of proving the continuity of the Lyapunov exponent. In this text we present a geometric version of the Avalanche Principle in the context of hyperbolic spaces, which will extend the usage of these techniques to study the drift in such spaces. This continuity criteria applies not only to the drift but also to the limit point of the process itself. We apply this abstract result to derive continuity of the drift for Markov processes in strongly hyperbolic spaces.
	\end{abstract}
	
	\section{Introduction}
	
	The study of random products of operators appears naturally in some problems of mathematics and its applications. In this text we take a closer look at the product of isometries of strongly hyperbolic metric spaces through their action on the space. In the prequel \cite{sampaio2021regularity} the author proved that for random walks, the drift, as a function of the measure, exhibits nice regularity properties such as large deviation estimates and continuity. In this text, where the context is strongly hyperbolic spaces, we extend the previous result by presenting a general condition for continuity, which also applies to Markov processes.
	
	We call this general condition the Abstract Continuity Theorem, as it allows us to obtain continuity in a quite abstract setting of any ergodic transformation picking our isometries, provided some conditions are met. Such a result is inspired by its linear counterpart for Lyapunov exponents by Duarte and Klein \cite{duarte2016lyapunov}. The key ingredients for our work will be large deviations and the Avalanche Principle. In particular, one of our goals is to prove that in the presence of subgaussian large deviation estimates for the drift, the drift is continuous.
    	
    Large deviations deal with the idea of rate of convergence associated with average limit quantities, whose value is independent of the path taken, 
    although the rate may vary. In this work we will require that the number of paths not converging fast enough decay exponentially over time.
	
	Notice that if the drift is positive, the action of the successive products of isometries on the space should not live in any compact set. In order to understand this behaviour we will look at sensible compactifications and boundaries at the space. Strongly hyperbolic spaces come equipped with a natural boundary, known as the Gromov, whilst as a metric space they admit a compactification by horofunctions. These two objects exhibit a deep relationship and allow us to state that, as a consequence of Karlsson and Gouëzel's theorem, (see \cite{gouezel2020subadditive}), paths of a  process with positive drift almost surely converge to a point in the Gromov boundary. We call this limit point of the process hitting point and its continuity will also be the object of our study.
	
	The Avalanche Principle, introduced in $\cite{goldstein2001holder}$, is a result which allows us to take conclusions on a product of operators from its factors. Together with large deviations, the principle allows us to push controls of finite nature, such as continuity and positivity forward in time. In this text we use it for both problems we've stated so far; continuity of the drift and hitting point. An avalanche principle for CAT($-1$) spaces was obtained previously by Oregon-Reyes \cite{oregon2020avalanche}. We present a new more succinct proof for strongly hyperbolic spaces, at a cost in the hypothesis.

	Strongly hyperbolic spaces are well behaved at the boundary when compared to purely Gromov spaces, which is a consequence of continuity of the Gromov product. Thus strong hyperbolicity is an essential step in our method for obtaining continuity, specially when it comes to the hitting point. We will see this become evident as the Avalanche Principle requires strong hyperbolicity.
	
	In this work we apply a spectral argument similar to Nagaev's \cite{nagaev1957some} and Duarte-Klein \cite{duarte2017Klein} to obtain large deviations in Markov systems. In the hyperbolic setting these methods have also been used to develop central limit theorems \cite{bjorklund2010central} as well as continuity of the drift for random walks \cite{aoun2021random, sampaio2021regularity}. Random walks in groups acting on hyperbolic space have been intensively studied and other regularity results are known, such as analiticity for finitely supported measures in hyperbolic groups \cite{gouezel2015analyticity} and large deviations principles for countable groups \cite{boulanger2020large}. The case of Markov systems is a lot less studied \cite{bougerol1999brownian, naor2006markov}. 
	
	In the remainder of the introduction we define all the mathematical objects at hand and display the full statements of the theorems. Section 2 is devoted to the Avalanche principle whilst sections 3 and 4 contain the proof of the abstract continuity theorem. To finish, in section 5 we obtain large deviations for Markov systems over these groups of isometries, which allows us obtain continuity.
	
	\subsection{Geometric setting}
	
	Let $X$ be a metric space, define the Gromov product in $X$ as
	\begin{equation*}
		\langle x \, , \, z \rangle_y := \frac{1}{2}\left( d(x,y) + d(y,z) - d(x,z) \right) \hspace{1cm} \forall x,y,z \in X.
	\end{equation*}
	We say that $X$ is a Gromov $\delta$-hyperbolic space, or simply hyperbolic space, if for every $x,y,z$ and $w$ in $X$,
	\begin{equation}
		\label{4pcondition}
		\langle x \, , \, z \rangle_w \geq \min\{ \langle x \, , \, y \rangle_w \, , \, \langle y \, , \, z \rangle_w \} - \delta.
	\end{equation}
	We call (\ref{4pcondition}) the 4-point condition of hyperbolicity or Gromov's inequality.
	
	A metric space $X$ is said to be geodesic if for every two points $x$ and $y$ in $X$, there exists an isometric embedding $\gamma:[0,d(x,y)]\to X$ connecting $x$ to $y$. For geodesic spaces, Gromov hyperbolicity has more geometric flavour (see \cite{das2017geometry}): $X$ is $\delta$-hyperbolic if there exists $\delta>0$ such that for every triangle in $X$, any side is contained in a $3\delta$-neighbourhood of the other two, in other words, geometrically, triangles are thin.
	
	We say that a sequence $(x_n)$ in an hyperbolic space $X$ with basepoint $x_0$ is a Gromov sequence if $\langle x_n\, , \, x_m \rangle_{x_0}$ tends to infinity as $m$ and $n$ tend to infinity. Two Gromov sequences $(x_n)$ and $(y_n)$ are equivalent, $(x_n)\sim (y_n)$, if $\langle x_n\, , \, y_n \rangle_{x_0}$ tends to infinity as $n$ tends to infinity. Gromov's inequality implies that this is an equivalence relation. The Gromov boundary, denoted by $\partial X$, is the set of equivalence classes of Gromov sequences.  Finally, $G$, the group of isometries of $X$ naturally acts on $\partial X$ by sending $\xi = [x_n]_\sim$ to $g\xi = [gx_n]_\sim$. 
	
	The Gromov product in $X$ may be extended to its Gromov boundary: given $\xi, \eta \in \partial X$ and $y,z\in X$, let
	\begin{align*}
		\langle \xi \, , \, \eta \rangle_z & := \inf \left\{\liminf_{n,m\to \infty}\, \langle x_n \, , \, y_m \rangle_z \, : \, (x_n)\in \xi, \, (y_m)\in \eta \right\},\\
		\langle x \, , \, \eta \rangle_z &= \langle \eta \, , \, x \rangle_z := \inf \left\{\liminf_{n\to \infty} \, \langle x_n \, , \, x \rangle_z \, : \, (x_n)\in \xi \right\}.
	\end{align*}
	
	Denote by $\bord X$ the set $X\cup \partial X$. Given $1 < b \leq 2^\frac{1}{\delta}$ and $x\in X$ consider the symmetric map $\rho_{x, b}: \bord X \times \bord X \to \mathbb{R}$ given by
	\begin{equation*}
		\rho_{x,b}(\xi\, , \, \eta) = b^{-\langle \xi \, , \, \eta \rangle_{x}}.
	\end{equation*} 
	
	\begin{definition}
	    We say that a metric space $X$ is a strongly hyperbolic space if for every $x\in X$ the map $\rho_{x,b}$ satisfies the triangle inequality, in particular $\rho_{x,b}$ defines a metric in $\partial X$.
	\end{definition}
	
	In particular, for every $\eta, \xi, \zeta \in \bord X$ we have the inequalities
	\begin{equation*}
	    \rho_{x,b}(\xi, \zeta) \leq \rho_{x,b}(\xi, \eta) +  \rho_{x,b}(\eta, \zeta) \leq 2 \max \{\rho_{x,b}(\xi, \eta), \rho_{x,b}(\eta, \zeta)\},
	\end{equation*}
	so applying the logarithm of base $b$ on both sides yields Gromov hyperbolicity with $\delta = \log_b 2$.
	
	Fixing a basepoint $x_0\in X$ we denote by $\overbar{D}_b$ the map $\rho_{x_0,b}$. Throughout the text, $X$ will denote a geodesic and separable although not necessarily proper strongly hyperbolic metric space.
	
	Strongly hyperbolic spaces are very well behaved at infinity, in particular the Gromov product is continuous in such spaces, in other words, for such spaces given $\xi, \eta \in \partial X$ and $z \in X$
	\begin{equation*}
		\langle \xi \, , \,\eta\rangle_z = \lim_{n\to \infty} \langle x_n \,,\, y_n \rangle_z
	\end{equation*}
	for every $(x_n)\in \xi$ and $(y_n)\in \eta$. Next we endow $\bord X$ with a metric.
	
	\begin{proposition}[Proposition 3.6.13 in \cite{das2017geometry}]
		For every $\xi, \eta \in \bord X$ let
		\begin{equation*}
			D_b(\xi, \eta) := \min \left\{ (\log b)d(\xi \, , \, \eta) \, ; \,  \overbar{D}_b(\xi \, , \, \eta) \right\},
		\end{equation*}
		using the convention $d(\xi, \eta)=\infty$ if either $\xi$ or $\eta$ belong to $\partial X$ and $\eta \neq \xi$. Then $D_b$ is a metric in $Bord X$, called the visual metric, inducing in $X$ the same topology as the metric $d$.
	\end{proposition}

	Since the Gromov product is always positive, we have $D_b \leq 1$, in particular, $\bord X$ is a bounded space when equipped with this metric. Its main drawback however is that $\bord X$ compact if and only if the space $X$ is locally compact. To combat this problem consider now the injection
	\begin{align*}
		\rho: X & \to C(X) \\
		x & \mapsto h_x( \cdot ) = d(\cdot , x) - d(x,x_0),
	\end{align*}
	where $C(X)\subset \mathbb{R}^X$ is the space of continuous functions in $X$ endowed with the topology of pointwise convergence. Then  $X^h := \overbar{\rho(X)}$ is compact. We call the elements of $X^h$ horofunctions and  $X^h$ the horofunction compactification of $X$. The action of $G$ on $X$ extends to its horofunction boundary as follows: for every $h \in X^h$, $g\in G$ and $z\in X$,
	\begin{equation}
	    \label{horoAction}
		g\cdot h(z)= h(g^{-1}z)-h(g^{-1}x_0).
	\end{equation}
	The horofunction compactification can be partitioned into its finite and infinite $G$-invariant parts given, respectively, by $X_F^h := \{h\in X^h \, : \, \inf(h)>-\infty \}$ and $X_\infty^h := \{h \in X^h \, : \, \inf(h)= -\infty \}$.
	
	The two boundaries are related by the local minimum map (see \cite{maher2018random}) $\phi:X^h \to \partial X$ sending every horofunction in $h\in X^h$ to the unique point $\xi \in \partial X$ such that for every $(x_n)\in \xi$,  $\lim_{n\to \infty}h(x_n) = - \infty$. This map continuous, surjective and $G-$equivariant map for general hyperbolic spaces. In the case  of strongly hyperbolic spaces, more can be said, in particular $\phi$ is a homeomorphism; again showcasing the good behaviour of such spaces at infinity. This will allow us to work interchangeably between the two boundaries. We this in mind, we denote by $h_\xi \in X^h$ the horofunction related to $\xi\in \partial X$ under the homeomorphism.
	
	At this point we equip $G$ with a metric that tracks the behaviour of its action $\bord X$. Namely, given $1<b\leq 2^{1/ \delta}$ we consider
	\begin{equation*}
		d_G(g_1,g_2) := \max \left\{ \sup_{\xi \in \bord X} D_b(g_1\xi, g_2\xi) \, ; \, \sup_{\xi \in \bord X} D_b(g_1^{-1}\xi, g_2^{-1}\xi) \right\}
	\end{equation*}
	for every $g_1, g_2 \in G$. In \cite{sampaio2021regularity} it was proven that $G$ with this metric is a topological group. 
	
	To finalize this introduction, we note  that the action of $G$ in $X$ satisfies a well understood mean value type formula.
	
	\begin{proposition}[Proposition 5 \cite{sampaio2021regularity}]
		\label{VisualMetric}
		Let $g\in G$ and $\xi, \eta \in \bord X$, then
		\begin{equation*}
			\frac{\overbar{D}_b(g\xi\, , \, g\eta)}{\overbar{D}_b(\xi\, , \, \eta)} 
			=b^{- \frac{1}{2} \left[ h_\xi(g^{-1}x_0) + h_\eta(g^{-1}x_0) \right]}.
		\end{equation*}
	\end{proposition}
	
	In \cite{sampaio2021regularity}, one finds the result as an inequality involving a constant $C(\delta)$ depending on $\delta$. For strongly hyperbolic spaces, the result is improved to obtain equality.
	
	\subsection{Dynamical setting}
	
	We will denote by $G$ the group of isometries of $X$. Let $(\Omega,\mu, \beta)$ be a standard probability space with measure $\mu$ and $\sigma$-algebra $\beta$, and $T:\Omega \to \Omega$ be an ergodic measure preserving transformation. We say that a measurable map $a:\mathbb{N}\times \Omega \to G$ is a multiplicative cocycle in $G$ over $T$ if $a(n+m, \omega) = a(n,\omega)a(m, T^n\omega)$. Given a Borel measurable $g:\Omega\to G$ consider its associated multiplicative cocycle
	\begin{equation*}
		a(n, \omega) = 	g^{(n)}(\omega)=g(\omega)g(T\omega)...g(T^{n-1}\omega),
	\end{equation*}
	for every $n\in \mathbb{N}$ and $\omega \in \Omega$. A cocycle is thus comprised of the information $(g, T, \Omega, \beta)$, whenever it is clear we denote it simply by $g$. 
	
	\begin{definition}[Integrable Cocycle]
		\label{integrableCoc}
		Let $x_0$ be a basepoint in $X$. We say that a cocycle $(g, T, \Omega, \beta)$ is integrable if
		\begin{equation*}
			\int_\Omega d(g(\omega)  x_0,x_0) d\mu(\omega) <\infty.
		\end{equation*}
	\end{definition}	
	One of the fundamental characteristics of an integrable cocycle is its drift
	\begin{equation*}
		\ell(g):= \lim_{n\to \infty} \frac{1}{n}\int_\Omega d(g^{(n)}(\omega)  x_0,x_0) d\mu(\omega) = \lim_{n\to \infty} \frac{1}{n}d(g^{(n)}(\omega) x_0,x_0),
	\end{equation*}
	where the first limit exists by Kingman's ergodic theorem whilst the second equality is true for almost every $\omega$ due to ergodicity, moreover none of these limits depend on the basepoint $x_0$.
	
	In this text we are interested in understanding the continuity of the drift as a function of $g$. This of course encompasses a problem, as we have to set a class of cocycles to work with. Let $G$ still stand for a group acting by isometries on a hyperbolic space $X$, $1<b\leq 2^\frac{1}{\delta}$ and $D_b$ the visual metric on $\partial X$. Given $g:\Omega \to G$ we denote by $g^{-1}:\Omega \to G$ the map that sends $\omega$ to $g(\omega)^{-1}$. Consider $S(\Omega, G)$ to be the space of measurable cocycles $g:\Omega \to G$ such that $g^{-1}$ is also measurable and 
	\begin{equation*}
	    d_\infty(g):=\sup_{\omega \in \Omega}b^{ d(g(\omega)x_0, x_0)}
	\end{equation*} 
	is finite. 
	
	Define the following pseudometric
	\begin{equation*}
		d_\infty(g_1,g_2) := \mathrm{ess}\, \mathrm{sup}_{\omega \in \Omega} d_G(g_1(\omega)\, , \, g_2(\omega)),
	\end{equation*}
	for every $g_1, g_2 \in S(\Omega, G)$. Define the equivalence relation
	\begin{equation*}
		g_1 \sim g_2  \Leftrightarrow d_\infty(g_1,g_2)=0
	\end{equation*}
	in $S(\Omega, G)$, so the set of equivalence classes $S^{\infty}(\Omega, G)$ becomes a metric space when equipped with $d_\infty$. We can now think of the drift as a map
	\begin{align*}
		\ell : S^{\infty}(\Omega, G) & \to \mathbb{R} \\
		g & \mapsto \ell(g).
	\end{align*}
	
	In \cite{gouezel2020subadditive}, Karlsson and Gouëzel prove that there is an horofunction that tracks the process $g^{(n)} x_0$ in any metric space and, provided the drift is positive, that horofunction belongs to $X_h^\infty$. In \cite{sampaio2021regularity} the author proved a more descriptive version for cocycles acting on Gromov hyperbolic spaces, namely:
	
	\begin{theorem}[Hyperbolic Multiplicative Ergodic Theorem]
		\label{hmet}
		Let $X$ be a separable geodesic Gromov hyperbolic space and $(g, T, \Omega, \beta)$ an integrable cocycle with positive drift. For almost every $\omega$ in $\Omega$ there is a filtration of the horofunction boundary
		\begin{equation*}
			X_-^h(\omega) \subset X_+^h(\omega) = X^h,
		\end{equation*} 
		such that:
		\begin{enumerate}
			\item for every $h\in X_+^h(\omega)\backslash X_-^h(\omega) $
			\begin{equation*}
				\lim_{n\to \infty} \frac{1}{n} h(g^{(n)}(\omega) x_0) = \ell(g);
			\end{equation*}
			\item for every $h \in X_-^h(\omega)$
			\begin{equation*}
				\lim_{n\to \infty} \frac{1}{n} h(g^{(n)}(\omega) x_0) = -\ell(g),
			\end{equation*}
			and given $h_1, h_2\in X_-^h $, one has $\sup_{z\in X} |h_1(z)-h_2(z)|<\infty$.
		\end{enumerate}
		Moreover the filtration is $G-$invariant, that is, 
		\begin{equation*}
			g(\omega)\cdot X_-^h(\omega) = X_-^h(T\omega)
		\end{equation*}
		and is measurable provided $\Omega$ is a standard probability space.
	\end{theorem}
	
	\begin{remark}
	    In the case of strongly hyperbolic spaces more can be said regarding $X_-^h$; in fact it consists of a single horofunction which is picked measurably, that is, the map $\omega \mapsto h_\omega^- \in X_-^h(\omega)$ is measurable. This fact will be very important later when we obtain large deviations for the Markov setting.
	\end{remark}

	The proof in \cite{sampaio2021regularity} was done only in the case of random walks but it is immediately adapted to the case of an ergodic base transformation.
	The intuition behind the result is based on looking at the sequence $h_{g^{(n)}(\omega)x_0}$ whose limit in $X^h$ is an horofunction belonging to $X_-^h(\omega)$. For horofunctions in $X_\infty^h$ one can prove the sequences that give rise to them are Gromov, moreover, all horofunctions in $X_-^h(\omega)$ come from equivalent Gromov sequences. This yields a well defined point $\xi(g, \omega)$ in $\partial X$ whose sequences yield horofunctions in $X_-^h(\omega)$. 
	
	Consider $S_+^\infty(\Omega, G)$ to be the subspace of $S^\infty(\Omega, G)$ consisting of the elements $g\in S^\infty(\Omega, G)$ with positive drift. Then by the previous theorem, for every $g\in S_+^\infty(\Omega, G)$  we can consider the almost everywhere defined partial map
	\begin{align*}
		\xi_g : \Omega & \to \partial X \\
		\omega &\mapsto \xi(g, \omega).
	\end{align*}
	By the end part of Theorem \ref{hmet}, since  $\Omega$ is standard, $\xi_g$ belongs to $S^1(\Omega, \partial X)$, the space of bounded measurable maps from $\Omega$ to $\partial X$ where we consider the metric
	\begin{equation*}
		d_1(f_1, f_2) := \int_\Omega \overbar{D_b}(f_1(\omega), f_2(\omega)) d\mu(\omega),
	\end{equation*}
	for every $f_1, f_2\in S^1(\Omega, \partial X)$. Finally we define the map
	\begin{align*}
		\xi : S_+^\infty(\Omega, G)  &\to S^1(\Omega, \partial X) \\
		g &\mapsto \xi_g.
	\end{align*}

	\subsection{Markov systems}

	Our main example where the result applies is Markov systems. We will begin by introducing the probabilistic language which we will use later, and then briefly present how to translate it into the dynamical language used previously through the Markov shift. Our presentation on the subject follows that of Duarte and Klein \cite{duarte2017Klein}.
	
	
	\begin{definition}[Markov Kernel]
		Let $\Gamma$ be a metric space and let $\mathcal{F}$ be its Borel $\sigma$-algebra. A Markov kernel is a function $K: \Gamma \times \mathcal{F} \to [0,1]$ such that
		\begin{enumerate}
			\item for every $\omega_0\in \Gamma$, $E \mapsto K(\omega_0,E)$ is a probability measure on $\Sigma$; 
			\item The mapping $\omega_0 \to K(\omega_0, \cdot)$ is continuous with respect to the weak-* topology in $\mathrm{Prob}(\Gamma)$.
			\item for every $E \in \mathcal{F}$, the function $\omega_0 \to K(\omega_0,E)$ is $\mathcal{F}$-measurable.
		\end{enumerate}
		
	\end{definition}
	
	A probability measure $\mu$ on $(\Gamma,\mathcal{F})$ is $K$-stationary if for every $E\in \mathcal{F}$,
	\begin{equation*}
		\mu(E) = \int_\Sigma K(\omega_0, E)\mu(d\omega_0).
	\end{equation*}
	A set $E\in \mathcal{F}$ is said to be $K$-invariant when $K(\omega_0, E) = 1$ for all $\omega_0 \in E$ and $K(\omega_0, E) = 0$ for all $\omega_0 \in \Gamma \backslash E$. A $K$-stationary measure $\mu$ is called ergodic when there is no K-invariant set $E \in \mathcal{F}$ such that  $0<\mu(E)<1$. Using the usual argument through Krein-Milman's theorem, ergodic measures are the extremal points in the convex set of $K$-stationary measures. A Markov system is a pair $(K,\mu)$, where $K$ is a Markov kernel on $(\Gamma, \mathcal{F})$ and $\mu$ is a $K$-stationary probability measure.
	
	Typically the considerations above are only done for compact  $\Gamma$ as this easily yields the existence of stationary measures, in this work however we will also need to work with non-compact spaces. Fortunately we will be able to find stationary measures for the non-compact cases that interest us.
	
	We can define the iterated Markov kernel inductively, setting $K^1 = K$ and 
	\begin{equation*}
		K^{n+1}(\omega_0, E) = \int_\Gamma K^n(\omega_1,E)K(\omega_0,d\omega_1),
	\end{equation*}
	for $n>1$.
	
	Given $(K,\mu)$ a pair formed by a Markov Kernel a not necessarily stationary measure $\mu \in \Prob(\Gamma)$, consider $\Omega = \Gamma^\mathbb{N}$ the space of sequences $\omega=(\omega_n)$ in $\Gamma$. The product space $\Omega$ is metrizable. Its Borel $\sigma$-algebra $\mathcal{B}=\mathcal{F}^\mathbb{N}$ is the product $\sigma$-algebra generated by the $\mathcal{B}$-cylinders, that is, generated by the sets
	\begin{equation*}
		C(E_0,...,E_m) := \{ \omega \in \Omega \, : \, \omega_j \in E_j, \textrm{ for } 0\leq j \leq m\},
	\end{equation*}
	where $E_0,...,E_m \in \mathcal{F}$.
	
	The set of $\mathcal{F}$-cylinders forms a semi-algebra on which
	\begin{equation*}
		\mathbb{P}_\mu [C(E_0,...,E_m)]:=\int_{E_m}... \int_{E_0} \mu(d\omega_0) \prod_{j=1}^{m} K(\omega_{j-1},d\omega_j).
	\end{equation*}
	defines a pre-measure. By Carathéodory's extension theorem, it extends to a measure, still denoted $\mathbb{P}_\mu$ and often called the Kolmogorov extension of $(K, \mu)$ on $(\Omega, \mathcal{B})$.
	
	Given a random variable $\zeta:\Omega \to \mathbb{R}$, its expected value with respect to $\mu$ in $(\Gamma, \mathcal{F})$ is denoted by
	\begin{equation*}
		\mathbb{E}_\mu(\zeta) :=\int_\Omega \zeta d\mathbb{P}_\mu.
	\end{equation*}
	If $\mu$ is $\delta_{\omega_0}$ the Dirac measure at $\omega$, then we soften the notation by setting $\mathbb{E}_{\omega_0} = \mathbb{E}_{\delta_{\omega_0}}$ as well as $\mathbb{P}_{\omega_0} = \mathbb{P}_{\delta_{\omega_0}}$.
	
	From construction, the sequence of random variables $e_n:\Omega \to \Sigma$, given by $e_n(\omega):=\omega_n$ for $\omega=(\omega_n)\in \Omega$, is a Markov chain with initial distribution $\mu$ and transition kernel $K$, that is, for every $\omega \in \Gamma$ and $E\in \mathcal{F}$,
	\begin{enumerate}
		\item $\mathbb{P}_\mu [e_0 \in E] = \mu(E),$
		\item $\mathbb{P}_\mu [e_n \in E \, | \, e_{n-1}=\omega_{n-1}] = K(\omega_{n-1},E)$.
	\end{enumerate}
	Moreover the process $\{e_n\}_{n\in \mathbb{N}}$ is stationary with respect to $(\Omega, \mathcal{F}, \mathbb{P}_\mu)$ if and only if $\mu$ is $K$-stationary.
	
	Consider the shift map $T:\Omega \to \Omega$, $T(\omega_n) = (\omega_{n+1})$. The shift $T$ is continuous and hence $\mathcal{B}$ measurable. Moreover $T$ preserves a measure $\mathbb{P}_\mu$ if and only if $\mu$ is $K-$stationary. We call the triplet $(\Omega, \mathbb{P}_\mu, T)$ a Markov shift.
	
	Suppose now that $\Gamma$ is compact, any continuous $g\in C(\Gamma \times \Gamma, G) \subset S^\infty (\Sigma \times \Sigma, G)$, where $S^\infty (\Sigma \times \Sigma, G)$ is the subspace of $S^\infty (\Omega, G)$, consisting of cocycles which depend only on the first two variables, defines a cocycle in $G$ over the Markov shift $(\Omega, \mathbb{P}_\mu, T )$, $a: \mathbb{N} \times \Omega \to G$ given by
	\begin{equation*}
		a(n,\omega) = g^{(n)}(\omega) := g(\omega_0,\omega_1)g(\omega_1,\omega_2) ... g(\omega_{n-1},\omega_n).
	\end{equation*}
	From this point on we will also omit the reference to the $\omega$'s in $g^{(n)}(\omega)$ whenever there is no room for confusion, by simply writing $g^{(n)}$

	\begin{definition}[Strongly Mixing]
		Let $B$ be a Banach space contained in $L^\infty(\Sigma)$. We say a Markov system $(K, \mu)$ is strongly mixing in $B$ if there are constants $C>0$ and $0<\sigma<1$ such that for every $f\in B$, all $x\in \Sigma$ and $n\in \mathbb{N}$,
		\begin{equation*}
			\left| \int_\Sigma f(\omega_1) K^n(\omega_0, d\omega_1) - \int_\Sigma f(\omega_1)\mu(d\omega_1)\right| \leq C\sigma^n||f||_B.
		\end{equation*}
	\end{definition}

	\subsection{Results}

	Let $X$ stand for a strongly hyperbolic metric space with basepoint $x_0$. Define the finite scale drift of $g\in S^\infty(\Omega, G)$ at time $n\in \mathbb{N}$ as
	\begin{equation}
	    \label{finiteScaleDrift}
		\ell_n(g):=\frac{1}{n}\int_\Omega d(g^{(n)}(\omega)  x_0,x_0) d\mu(\omega),
	\end{equation}
	which clearly satisfies $\ell_n(g) \to \ell(g)$ as $n$ goes to $\infty$. 
	
	Henceforth we fix $\mathcal{C}\subset S^\infty(\Omega, G)$ a class of cocycles equipped with some distance $d_\mathcal{C}$ such that $d_\mathcal{C}(g_1, g_2) \geq d_\infty(g_1, g_2)$. In some cases we place additional restrictions on our cocycles besides simply belonging to $S^\infty(\Omega, G)$; that is where the proving the results for smaller classes may prove valuable. In the same spirit we will denote by $\mathcal{C}_+$ the set $\mathcal{C} \cap S_ +^\infty(\Omega, G)$.
	
	\begin{definition}[Large deviation estimates]
	    \label{LDE}
		Fix $x_0 \in X$. A cocycle $g\in \mathcal{C}$ is said to satisfy a uniform large deviation estimates of exponential type if there are constants $r>0, \, c>0$ and for every $\varepsilon>0$ there exists $\overbar{n}=\overbar{n}(\varepsilon)$ such that
		\begin{equation*}
			\mu\left\{ \omega \in  \Omega \, : \, \left| \frac{1}{n} d(g_1^{(n)}(\omega)x_0, x_0)- \ell_n(g_1)\right|>\varepsilon \right\} < b^{-c\varepsilon^2n}
		\end{equation*}
		for every $g_1 \in \mathcal{C}$ with $d_\mathcal{C}(g, g_1)<r$ and every $n\geq \overbar{n}$.
	\end{definition}
	Our main goal in this text is to prove an abstract continuity theorem for the drift, provided large deviation estimates are present.
	
	\begin{theorem}
		\label{ContTheorem}
		Let $(T, \Omega, \mu, \beta)$ be an ergodic  measure preserving dynamical system. Suppose every $g\in \mathcal{C}_+$ satisfies a uniform large deviation estimate, then
		\begin{enumerate}
			\item[1)] The drift $\, \ell:\mathcal{C} \to \mathbb{R}$ is continuous;
			\item[2)] The drift $\, \ell:\mathcal{C}_+ \to \mathbb{R}$ is locally  Hölder continuous;
			\item[3)] Moreover, $\xi: \mathcal{C}_+ \to S^1(\Omega, \partial X)$ is locally Hölder continuous.
		\end{enumerate}
	\end{theorem}
	
	The idea of the proof is to obtain a quantitative modulus of continuity at finite time for the maps $\ell_{n_0}(g)$ and then transport these controls to a forward time $n_1=n\, n_0$ where $n\in \mathbb{N}$. To perform this transport we will use the uniform large deviation estimates together with the following theorem which allows us relate the displacement of a product of isometries with the displacements of its terms.
	
	\begin{theorem}[Avalanche Principle]
		\label{MAP}
		Let $X$ be a strongly hyperbolic space, $x_0,...,x_{n}$ be a sequence of points in $X$ and $\rho, \sigma >0$ constants such that
		\begin{itemize}
			\item[G)] $d(x_{i-1},x_{i}) \geq \rho, \hspace{0.5cm} i=1,...,n$;
			\item[A)] $\langle x_{i-1}, x_{i+1}\rangle_{x_i} \leq \sigma, \hspace{0.5cm} i=1,...,n-1$;
			\item[P)] $2\sigma < \rho-2\delta$;
		\end{itemize}
		Then,
		\begin{itemize}
			\item [1)] $\langle x_{0}, x_{n}\rangle_{x_{n-1}} < \sigma + \frac{1}{\log b} b^{2\sigma-\rho+2\delta}$,
			\item [2)] $d(x_0, x_n) > \rho + (n-1)(\rho - 2\sigma - 2\delta )$,
			\item [3)] and the following inequality holds $$\left| d(x_0, x_{n}) + \sum_{i=2}^{n-1}d(x_{i-1},x_i) - \sum_{i=1}^{n-1} d(x_{i-1},x_{i+1}) \right| \leq 2(n-1)\frac{1}{\log b} b^{2\sigma-\rho+2\delta}.$$
		\end{itemize}
	\end{theorem}
	
	For CAT($-1$) spaces, condition $P)$ may be replaced with $\sinh(\rho-\sigma)>2\sinh(\rho/2)$, which is more general, specially for small values of $\rho$ (see \cite{oregon2020avalanche}). Our version applies to more general spaces and suffices for our applications.
	
	By the end of the text we will prove that Markov systems satisfy large deviation estimates provided some mild conditions are met, thus proving continuity of the drift in this setting. 
	
	We say that a cocycle $g\in S^\infty(\Sigma \times \Sigma, G)$ is irreducible with respect to $(K,\mu)$ if there is no measurable map $H: \Sigma \to X^h$ such that
	\begin{equation*}
	    g(\omega_{n-1}, \omega_n)H(\omega_{n-1}) = H(\omega_n)
	\end{equation*}
	for $\mathbb{P}_\mu$-almost every  $\omega$.
	
	\begin{theorem}
		\label{LDT}
		Let $\Sigma$ be a compact metric space, $(K, \mu)$ be a strongly mixing Markov system over $\Sigma$ and $g\in S^\infty(\Sigma \times \Sigma, G)$ be a continuous cocycle with positive drift which is irreducible with respect to $(K, \mu)$. Then $g$ satisfies uniform large deviations estimates in the class of irreducible continuous cocyles.
	\end{theorem}
	
	As a consequence of Theorems \ref{ContTheorem} and \ref{LDT}, we obtain the following Corollary
	
	\begin{corollary}
		\label{continuity}
		Let $\Sigma$ be a compact metric space, $(K, \mu)$ be a strongly mixing Markov system over $\Sigma$ and $\mathcal{I}(K) \subset S^\infty(\Sigma \times \Sigma, G)$ be the class  of continuous cocycles which are irreducible with respect to $(K, \mu)$. 
	\end{corollary}
	
	\begin{enumerate}
			\item[1)] The drift $\, \ell:\mathcal{I}(K) \to \mathbb{R}$ is continuous;
			\item[2)] The drift $\, \ell:\mathcal{I}(K)_+ \to \mathbb{R}$ is locally  Hölder continuous;
			\item[3)] Moreover, $\xi: \mathcal{I}(K)_+ \to S^1(\Omega, \partial X)$ is locally Hölder continuous.
		\end{enumerate}

	\section{Avalanche Principle}

	Before we tackle the proof, let us make two remarks; first that the hypothesis imply
	\begin{equation}
		\label{remProof1}
		\langle x_{i-1}, x_{i+1}\rangle_{x_i} + \langle x_{i}, x_{i+2}\rangle_{x_{i+1}} \leq 2\sigma < \rho - 2\delta \leq d(x_i,x_{i+1}) - 2\delta,
	\end{equation}
	secondly, that the left-hand side of the conclusion may be rewritten as
	\begin{equation}
		\label{remProof2}
		\left| d(x_0, x_{n}) - \sum_{i=1}^{n}d(x_{i-1},x_i) + 2 \sum_{i=1}^{n-1} \langle x_{i-1}, x_{i+1}\rangle_{x_i}\right|.
	\end{equation}
	
	\begin{proof}[ Proof of Theorem \ref{MAP}]
		We will base the proof in establishing two simple claims. 
		
		\textbf{Claim 1:}
		\begin{equation*}
			|\langle x_0 , x_k \rangle_{x_{k-1}} - \langle x_{k-2}, x_k \rangle _{x_{k-1}}|\leq \delta.
		\end{equation*}
		Let us use induction: The case $k=2$ is trivial. For $k>2$, notice that
		\begin{align*}
			\langle x_0 , x_{k-2} \rangle_{x_{k-1}} & = d(x_{k-1}, x_{k-2}) - \langle x_0 , x_{k-1} \rangle_{x_{k-2}} \\
			& \geq d(x_{k-1}, x_{k-2}) - \langle x_{k-3} , x_{k-1} \rangle_{x_{k-2}} - \delta && \textrm{by induction,}\\
			& > \langle x_{k-2} , x_{k} \rangle_{x_{k-1}} + \delta && \textrm{by (\ref{remProof1}).}
		\end{align*}
		Proceeding with the definition of hyperbolicity
		\begin{equation*}
			\langle x_{k-2} , x_{k} \rangle_{x_{k-1}}\geq \min \{\langle x_0 , x_{k-2} \rangle_{x_{k-1}} , \langle x_k , x_{0} \rangle_{x_{k-1}} \} - \delta,
		\end{equation*}
		where the minimum must be $\langle x_k , x_{0} \rangle_{x_{k-1}}$, otherwise we would get $\langle x_{k-2} , x_{k} \rangle_{x_{k-1}} > \langle x_{k-2} , x_{k} \rangle_{x_{k-1}}$. Whence
		\begin{equation*}
			\langle x_0 , x_{k} \rangle_{x_{k-1}} \leq \langle x_{k-2} , x_{k} \rangle_{x_{k-1}} + \delta.
		\end{equation*}
		Changing the roles of $x_0, x_{k-2}$ we get the claim. Point 2) in the Avalanche principle is an immediate consequence of this claim.

		\textbf{Claim 2:} Our second claim implies $1)$,
		\begin{equation*}
			|\langle x_0 , x_k \rangle_{x_{k-1}} - \langle x_{k-2}, x_k \rangle _{x_{k-1}}|\leq  \frac{1}{\log b }b^{2\sigma-\rho + 2\delta}.
		\end{equation*}
		Since $X$ is strongly hyperbolic, 
		\begin{equation*}
		    \left| b^{-\langle x_0, x_k \rangle _{x_{k-1}}} -  b^{-\langle x_{k-2}, x_k \rangle_{x_{k-1}}} \right| \leq b^{-\langle x_0, x_{k-2} \rangle_{x_{k-1}}}
		\end{equation*}
		which together with Lagrange's mean value theorem with $f(x)=b^{-x}$, followed by claim 1, and the inequality $\langle x_0, x_{k-2} \rangle_{x_{k-1}} \geq d(x_{k-1}, x_{k-2})-\langle x_{k-3}, x_{k-1} \rangle_{x_{k-1}} - \delta$ obtained in claim 1, yields
		\begin{align*}
			\left| \langle x_0 , x_k \rangle_{x_{k-1}} - \langle x_{k-2}, x_k \rangle_{x_{k-1}} \right|
			& \leq \frac{1}{\log b }b^{ \max \{ \langle x_0 , x_k \rangle_{x_{k-1}} , \langle x_{k-2} , x_k \rangle_{x_{k-1}} \}}\left| b^{-\langle x_0, x_k \rangle _{x_{k-1}}} -  b^{-\langle x_{k-2}, x_k \rangle_{x_{k-1}}} \right| \\
			& \leq \frac{1}{\log b}b^{\sigma + \delta} b^{-\langle x_{0}, x_{k-2} \rangle _{x_{k-1}}} \\
			& \leq \frac{1}{\log b } b^{\sigma + \delta} b^{\langle x_{k-3}, x_{k-1} \rangle _{x_{k-2}} - d(x_{k-1}, x_{k-2}) + \delta }\\
			& \leq \frac{1}{\log b} b^{2\sigma-\rho + 2\delta}.
		\end{align*}
        
		These claims were motivated by the relation
		\begin{equation*}
			d(x_0,x_n) = d(x_0,x_{n-1}) + d(x_{n-1}, x_n) - 2\langle x_0 , x_n \rangle_{x_{n-1}}.
		\end{equation*}
		We can now apply the analogue relation to $d(x_0,x_{n-1})$ to obtain that (\ref{remProof2}) is bounded above by  $(n-1)|\langle x_0 , x_n \rangle_{x_{n-1}} - \langle x_{n-2}, x_n \rangle _{x_{n-1}}|$. Together with claim 2, this concludes the proof.
		
	\end{proof}
	
	\begin{example}
		\label{Example1}
		Let us look at the hyperbolic plane $\mathbb{H}^2$. The hyperbolic plane is strongly hyperbolic with $b=e$ and $d(gx_0, x_0) = 2\log||g||$. Consider $g_0,..., g_{n-1}\in SL(2,\mathbb{R})$ isometries of $\mathbb{H}^2$. Finally take $x_0=i$ and $x_j = g^{(j)}\cdot i = g_0g_1...g_{j-1} \cdot i$. 
		Then the hypothesis read as follows
		\begin{itemize}
			\item[G)] $d(x_{j-1},x_{j})\geq \rho \Leftrightarrow 2\log ||g_{j-1}|| \geq \rho \Leftrightarrow ||g_{j-1}||^2 \geq e^\rho = \mu;$
			\item[A)] $\langle x_{j-1}, x_{j+1}\rangle_{x_j} \leq \sigma  \Leftrightarrow \frac{||g_{j-1}g_j||}{||g_{j-1}||\, ||g_{i}||} \geq e^{-\sigma} = \nu$;
			\item[P)] $\mu^{-1} < e^{-2\delta} \nu^{2}$,
		\end{itemize}
		whilst the conclusion reads
		\begin{equation*}
			\left| \log||g^{(n)}|| + \sum_{j=2}^{n-1} \log ||g_{j-1}|| - \sum_{j=1}^{n-1} \log ||g_{j-1}g_j|| \right| \leq 2(n-1)e^{2\sigma-\rho+2\delta} = 2e^{2\delta} (n-2) \frac{1}{\mu \nu^2}.
		\end{equation*}
		Upon taking transposes, we obtain a restatement of the $\SL(2,\mathbb{R})$ version of the Avalanche principle.
	\end{example}
	
	\section{Continuity of the Drift}
	
	In this section we prove the first assertion of Theorem \ref{ContTheorem}. This is done by following a specific route where we start by proving the continuity at a finite scale, then we transport the control to larger scales by an inductive step based on the Avalanche principle and the existence of large deviation estimates.

	\subsection{Finite Scale Continuity}
	
	Let us start by proving that at a finite scale the drift is continuous as well as understand  this continuity rate,  this is necessary for the next step where we try to transport these controls forward in time.
	
	\begin{lemma}
		Given $C>0$, set $G_C = \{g\in G \, : \, d(gx_0, x_0)< \log_b C\}$. The map $G_C\to \mathbb{R}$ defined by $g\mapsto d(gx_0, x_0)$ is Lipschitz continuous.
	\end{lemma}
	
	\begin{proof}
		Let $g_1, g_2\in G_C$. Notice that if $|d(g_1x_0,x_0)-d(g_2x_0,x_0)|\leq d(g_1x_0, g_2x_0)$. If $D_b(g_1x_0, g_2x_0) = (\log b) d(g_1x_0, g_2x_0) $ we are done, otherwise use the inequality $(\log b)x<b^{x/2}$,
		\begin{align*}
		    |d(g_1x_0,x_0)-d(g_2x_0,x_0)| & \leq d(g_1x_0, g_2x_0) \\
		    & \leq \frac{1}{\log b} b^{d(g_1x_0, g_2x_0)/2}\\
		    & = \frac{b^{d(g_1x_0,x_0)/2 + d(x_0, g_2x_0)/2}}{\log b} b^{d(g_1x_0, g_2x_0)/2-d(g_1x_0,x_0)/2 - d(x_0, g_2x_0)/2} \\
		    & \leq \frac{C}{\log b} b^{-\langle g_1x_0 \, , \, g_2x_0 \rangle_{x_0}} \leq D_b(g_1x_0, g_2x_0) \leq D_G(g_1, g_2).
		\end{align*}
		which concludes the proof.
	\end{proof}
	
	The technique associated with inequality obtained in the previous lemma will be used multiple times throughout the text.
	
	\begin{lemma}
		\label{normIneq}
		Let $g\in S^\infty(\Omega, G)$, there exist $C=C(g)>0$ and $r>0$ such that if $g_1, g_2 \in S^\infty(\Omega, G)$ with $d_\infty(g_i,g)<r$ for $i=1,2$, then for every $n \in \mathbb{N}$ and for every $\omega \in \Omega$
		\begin{enumerate}
			\item $d_\infty(g_1) < C$;
			\item $d_G(g_1^{(n)}(\omega), g_2^{(n)}(\omega)) \leq \,  nC^{n-1}d_\infty(g_1\, , \, g_2)$.
		\end{enumerate}
	\end{lemma}
	
	\begin{proof}
		Point $1.$ is a consequence of the previous lemma. Denote by $T:\Omega \to \Omega$ the ergodic transformation at hand. For every  $\omega \in \Omega$, one has (see in \cite{sampaio2021regularity} at the end of the proof of Theorem 2, recall that $C(\delta)=1$ in our setting)
		\begin{align*}
			d_G(g_1^{(n)}(\omega), g_2^{(n)}(\omega)) & \leq d_G(g_1(\omega), g_2(\omega)) + b^{d(g_1(\omega)x_0, x_0)}d_G(g_1^{(n-1)}(T\omega), g_2^{(n-1)}(T\omega)) \\
			&  \leq  d_\infty(g_1,g_2) + C d_G(g_1^{(n-1)}(T\omega), g_2^{(n-1)}(T\omega))
		\end{align*}
		so the claim follows by induction.
	\end{proof}
	
	\begin{proposition}[finite scale continuity]
		\label{FiniteScaleContinuity}
		Let $g\in S^\infty(\Omega, G)$. For every $g_1, g_2 \in S^\infty(\Omega, G)$ and for almost every $\omega \in \Omega$ there exists $C=C(g)>0$,
		\begin{equation*}
			\left|\frac{1}{n}d(g_1^{(n)}(\omega)x_0, x_0) - \frac{1}{n}d(g_2^{(n)}(\omega)x_0, x_0) \right|\leq  \frac{C^{n}}{\log b } d_\infty(g_1, g_2) \leq \frac{ b^{C_1 n}}{\log b }d_\infty(g_1\, , \, g_2).
		\end{equation*}
		where $C_1 := \log_b\left(C \right)$, in particular,
		\begin{equation*}
			\left|\ell_{n}(g_1)-\ell_n(g_2)\right| < \frac{ b^{C_1 n}}{\log b } d_\infty(g_1, g_2).
		\end{equation*}
	\end{proposition}
	
	Recall that $\ell_n$ stands for the finite scale drift, hence the bottom inequality in the proposition follows from the upper one after integration on $\omega$.
	
	\begin{proof}
		To soften notations, let us omit $\omega$ throughout the proof. 
		\begin{align*}
			\left|d(g_1^{(n)}x_0, x_0) - d(g_2^{(n)}x_0, x_0) \right| & \leq d(g_1^{(n)}x_0, g_2^{(n)}x_0) \\
			& \leq \frac{1}{\log b} b^{d(g_1^{(n)}x_0, g_2^{(n)}x_0)/2} \\
			& = \frac{b^{d(g_1^{(n)}x_0, x_0)/2 + d(g_2^{(n)}x_0, x_0)/2}}{\log b } b^{-\langle g_1^{(n)}x_0 \, , \, g_2^{(n)}x_0 \rangle_{x_0}} \\
			& \leq \frac{C^n}{\log b}d_G(g_1^{(n)}(\omega), g_2^{(n)}(\omega)) \\
			& \leq n \frac{C^{2n}}{\log b } d_\infty(g_1,g_2)
		\end{align*}
		Which concludes the proof.
	\end{proof}
	
	This proposition implies the continuity of the maps $\ell_n$. Since the drift $\ell(g)$ may be given as $\inf_{n\geq 1} \ell_n(g)$, the upper semi-continuity of $\ell(g)$ follows from the following lemma.
	
	\begin{lemma}
	    \label{USC}
		Let $M$ be a metric space and $f_n:M\to \mathbb{R}$ be a sequence of upper semi-continuous functions. Then, $f(x) = \inf_{n \geq 1} f_n(x)$, the pointwise infimum of these functions, is upper semi-continuous.
	\end{lemma}
	
	\begin{proof}
		Let $x\in M$ and take $\inf_{n \geq 1} f_n(x) = g(x) < r$, there must be $i\geq 1$ such that $f_i(x) < r$. Since $f_i(x)$ is upper semi-continuous, there must be a neighbourhood $V$ of $x$ such that for every $y \in V$ one has $f_i(y) < r$. Since $ g(y)\leq  f_i(y)$ for every $y$, we obtain $g(y)<r$ for every $y\in U$ thus proving the Lemma.
	\end{proof}
	
	Since $\ell$ is upper semi-continuous, it is continuous in the neighbourhood of the cocycles $g\in \mathcal{C}$ in which it is zero. With that said we focus cocycles in $\mathcal{C}_+$, where we obtain a stronger modulus of continuity.
	
	\subsection{Inductive Step}
	
	In this section we will understand how to pass the previously estabilished  controls forward through an inductive step based on the  large deviations estimates and the Avalanche principle. From this point on in the text we will use the notation $a \lesssim b$ to convey that there exists a universal constant $C$ such that $a\leq Cb$.

	\begin{lemma}
		Let $g\in S^\infty(\Omega, G)$, if $n, n_0, n_1, r \in \mathbb{N}$ are such that $n_1=n\, n_0 +r$ where $0\leq r < n_0$, then
		\begin{equation*}
			-2\log_b(C)\frac{n_0}{n_1}+ \ell_{(n+1)n_0}(g) \leq \ell_{n_1}(g) \leq \ell_{n\, n_0}(g)+2\log_b(C)\frac{n_0}{n_1}.
		\end{equation*}
	\end{lemma}
	
	\begin{proof}
		Given $n_1=n\, n_0 +r$ where $0\leq r < n_0$ we have, for every $\omega$, $g^{(n_1)}(\omega) = g^{(n\, n_0)}(\omega) g^{(r)}(T^{n\, n_0}\omega)$, whence
		\begin{equation*}
			d(g^{(n_1)}(\omega)x_0\, , \, x_0) \leq d(g^{(n\, n_0)}(\omega)x_0\, , \, x_0) + d(g^{(r)}(T^{n\, n_0} \omega)x_0\, , \, x_0),
		\end{equation*}
		integrating both sides, one has
		\begin{equation*}
			\ell_{n_1}(g) \leq \frac{n\, n_0}{n_1} \ell_{n\, n_0}(g) + \frac{r}{n_1}\ell_r(g).
		\end{equation*}
		which gives
		\begin{equation*}
			\ell_{n_1}(g) \leq \ell_{n\, n_0}(g) + \frac{r}{n_1}\left[\ell^{(r)}(g)-\ell^{(n\, n_0)}(g)\right] \leq \ell_{n\, n_0}(g) + 2\log_b(C)\frac{r}{n_1}.
		\end{equation*}
		
		For the leftmost inequality write $n_1 = (n+1)\, n_0 + q$ where $q = r-n_0$ {check again} and proceed similarly.
	\end{proof}
	
	The following proposition is the important step towards proving continuity of the drift. Its content is that if we obtain some control for time $n_0$, then we can transport it to time $n_1$ larger than $n_0$. To do this we break the orbit at time $n_1$ into smaller pieces of size $n_0$ which we then relate back with the larger piece of size $n_1$ by using the avalanche principle.
	
	\begin{proposition}[Inductive step]
		Let $g\in \mathcal{C}_+$ and $c, \overbar{n}$ be the uniform large deviation parameters. Fix $\varepsilon = \ell(g)/100 >0$ and denote $c_1:=\frac{c}{2}\varepsilon^2$. There are constants $C=C(g)>0, r = r(g)>0$, $\overbar{n_0} = \overbar{n_0}(g)\in \mathbb{N}$, such that for any $n_0>\overbar{n_0}$, if the inequalities
		\begin{align*}
			& \ell_{n_0}(g_1) - \ell_{2n_0}(g_1)< \eta_0 \\
			& |\ell_{n_0}(g_1) - \ell_{n_0}(g)|< \theta_0
		\end{align*}
		holds for any $g_1\in \mathcal{C}$ such that $d(g_1, g)<r$ and if the positive numbers $\eta_0, \theta_0$, satisfy 
		\begin{equation*}
			\theta_0 + 2\eta_0 < \ell(g) - 4 \varepsilon,
		\end{equation*}
		then for every $n_1$ such that  $ |n_1 - e^{c_1n_0}|<1 $ one has
		\begin{equation}
			\label{inductiveCont1}
			|\ell_{n_1}(g_1) + \ell_{n_0}(g_1) - 2\ell_{2n_0}(g_1) |\leq C\frac{n_0}{n_1}
		\end{equation}
		Furthermore,
		\begin{align}
			\label{inductiveCont2}
			&\ell_{n_1}(g_1) - \ell_{2n_1}(g_1)< \eta_1 \\
			\label{inductiveCont3}
			& |\ell_{n_1}(g_1) - \ell_{n_1}(g)|< \theta_1
		\end{align}
		where
		\begin{align*}
			&\theta_1 = \theta_0 + 4\eta_0 + C\frac{n_0}{n_1} \\
			&\eta_1 = C\frac{n_0}{n_1}.
		\end{align*}
	\end{proposition}

	\begin{proof}
		Throughout the proof $C$ will stand for some constant which isn't a priori always the same. We start the proof with some assumptions, in particular, making $r$ smaller if necessary, every $g_1$ with $d_\infty(g, g_1)<r$ satisfies large deviation estimates. We can also assume $\overbar{n_0}$ to be large enough so that $|\ell_n(g)-\ell(g)|<\varepsilon$ for $n \geq \overbar{n_0}$ which comes from the fact $\ell_n(g)$ converges to $\ell(g)$.
		
		With that said, let $g_1$ be in the conditions above. Assume $n_1=n\, n_0$ as otherwise we obtain an extra error of order $n_0/n_1$ which, by the previous lemma, is along the size of our control. Fix $x_0$ a basepoint in $X$ and define, for every $0\leq i \leq n-1$, the sequence of points
		\begin{equation*}
			x_i(\omega):= g_1^{(n_0)}(\omega)g_1^{(n_0)}(T^{n_0}\omega)\,...\,g_1^{(n_0)}(T^{(i-1)\, n_0}\omega)\,x_0,
		\end{equation*} 
		so that $x_{n} = g^{(n_1)}(\omega)x_0$ and for every $1 \leq i \leq n-1$,
		\begin{align*}
			d(x_i, x_{i-1}) & =  d(g_1^{(n_0)}(T^{(i-1)\, n_0}\omega)\, x_0\, , \,x_0) ,\\
			d(x_{i-1}, x_{i+1}) & = d(g_1^{(n_0)}(T^{(i-1)\, n_0}\omega) g_1^{(n_0)}(T^{(i)\, n_0}\omega)\, x_0 \, , \, x_0) = d(g_1^{(2n_0)}(T^{(i-1)\, n_0}\omega) x_0 \, , \, x_0).
		\end{align*} 
		
		At this point we are going to use the large deviation  estimates to verify the conditions of the avalanche principle are satisfied, with effect for every $m>\overbar{n_0}$ there exists a set $\mathcal{B}_{m}$ whose measure does not exceed $e^{-c\varepsilon^2m}$ such that for every $\omega \notin \mathcal{B}_{m}$ 
		\begin{equation*}
			-\varepsilon \leq  \frac{1}{m}d(g^{(m)}x_0, x_0)- \ell_m(g)  \leq \varepsilon
		\end{equation*}
		in particular, if $\omega \notin \mathcal{B}_{n_0}$ 
		\begin{align*}
			\frac{1}{n_0}d(x_1, x_0) & = 
			\frac{1}{n_0}d(g_1^{(n_0)}(\omega)x_0, x_0) \\
			& \geq \ell_{n_0}(g_1) - \varepsilon \\
			& > \ell_{n_0}(g) - \theta_0 - \varepsilon\\
			& \geq \ell(g) - \theta_0 - \varepsilon,
		\end{align*}
		whence,
		\begin{equation*}
			d(g_1^{(n_0)}(\omega)x_0, x_0) > n_0(\ell(g) - \theta_0 - \varepsilon) =: \rho_0 .
		\end{equation*}
		
		Through the same process we obtain for every $\omega\notin \mathcal{B}_{2n_0}$
		\begin{equation*}
			\frac{1}{2n_0}d(g_1^{(2n_0)}(\omega)x_0, x_0) \geq \ell^{(2n_0)}(g_1) - \varepsilon
		\end{equation*}
		as well as
		\begin{align*}
			\frac{1}{n_0}d(g_1^{(n_0)}(\omega)x_0, x_0) & \leq \ell_{n_0}(g_1) + \varepsilon \\
			\frac{1}{n_0}d(g_1^{(n_0)}(T^{n_0}\omega)x_0, x_0) & \leq \ell_{n_0}(g_1) + \varepsilon,
		\end{align*}
		for every $\omega\notin \mathcal{B}_{n_0}\cup T^{-n_0}\mathcal{B}_{n_0}$. Hence, for every $\omega \notin \mathcal{B}_{2n_0}\cup \mathcal{B}_{n_0}\cup T^{-n_0}\mathcal{B}_{n_0}  $
		\begin{align*}
			\left\langle x_0 \, , \, x_2 \right\rangle_{x_1} & = 
			\left\langle x_0 \, , \, g_1^{(2n_0)}(\omega)x_0 \right\rangle_{g_1^{(n_0)}(\omega)x_0} \\
			& =\frac{1}{2}\left( d(g_1^{(n_0)}(\omega)x_0, x_0) +  d(g_1^{(n_0)}(T^{n_0}\omega)x_0, x_0)  -d(g_1^{(2n_0)}(\omega)x_0, x_0)\right) \\
			&\leq n_0(\ell_{n_0} - \ell_{2n_0} + 2\varepsilon),
		\end{align*}
		in oher words, 
		\begin{equation*}
			\left\langle x_0 \, , \, x_2 \right\rangle_{x_1} <   n_0(\eta_0 + 2\varepsilon) =: \sigma_0.
		\end{equation*}
		
		Similar computations yield the same controls for every $1 \leq i \leq n-1$, under appropriate assumptions. Moreover, by hypothesis, $2\sigma_0 - \rho_0 = n_0( \eta_0+3\varepsilon + \theta_0 - \ell(g)) \leq -\varepsilon n_0$ so choosing $n_0$ large enough so that $-\varepsilon n_0<-2\delta$, the AP applies outside the set $\mathcal{B}_{n_0}^* = \cup_{i=0}^{n-1}T^{in_0}\mathcal{B}_{n_0}$ where we obtain the control
		\begin{equation*}
			\left| d(x_0, x_{n}) + \sum_{i=2}^{n-1}d(x_{i-1},x_i) - \sum_{i=1}^{n-1} d(x_{i-1},x_{i+1}) \right| \leq 2(n-1)\frac{1}{\log(b)} b^{2\sigma_0-\rho_0+2\delta},
		\end{equation*}
		which translates to
		\begin{equation*}
			\left| d(x_0, g^{(n_1)}(\omega)x_0) + \sum_{i=2}^{n-1}d(g_1^{(n_0)}(T^{(i-1)\, n_0}\omega)\, x_0\, , \,x_0) - \sum_{i=1}^{n-1} d(g_1^{(2n_0)}(T^{(i-1)\, n_0}\omega) x_0 \, , \, x_0) \right|  \lesssim n b^{-\varepsilon n_0}.
		\end{equation*}
		Dividing both sides by $n_1= n\, n_0$, one now obtains 
		\begin{align*}
			\bigg| \frac{1}{n_1}d(x_0, g^{(n_1)}(\omega)x_0) 
			& + \frac{1}{n}\sum_{i=2}^{n-1} \frac{1}{n_0}d(g_1^{(n_0)}(T^{(i-1)\, n_0}\omega)\, x_0\, , \,x_0) \\
			& - \frac{2}{n}\sum_{i=1}^{n-1} \frac{1}{2n_0} d(g_1^{(2n_0)}(T^{(i-1)\, n_0}\omega) x_0 \, , \, x_0) \bigg|  
			\lesssim b^{-\varepsilon n_0}.
		\end{align*}
    		Let $f(\omega)$ denote the bounded function on the left side. Notice that, for every $\omega \notin \mathcal{B}_{n_0}^\ast$, $|f(\omega)| \lesssim b^{-\varepsilon n_0}$, while in $\mathcal{B}_{n_0}^\ast$ the control $|f(\omega)| \leq C$ remains valid for some $C=C(g)$ since $g_1\in \mathcal{C}$. On the other hand,
		\begin{equation*}
			\int_\Omega f(\omega) d\mu (\omega) = \ell_{n_1}(g_1) +\frac{n-2}{n}\ell_{n_0}(g_1) - \frac{2(n-1)}{n}\ell_{2n_0}(g_1),
		\end{equation*}
		hence 
		\begin{align*}
			\bigg|\ell_{n_1}(g_1) & +\frac{n-2}{n}\ell_{n_0}(g_1) - \frac{2(n-1)}{n}\ell_{2n_0}(g_1) \bigg| \leq \int_\Omega |f(\omega)| d\mu(\omega) \\ 
			& \int_{\Omega \backslash \mathcal{B}_{n_0}^\ast}|f(\omega)| d\mu(\omega)  + \int_{ \mathcal{B}_{n_0}^\ast}|f(\omega)| d\mu(\omega) \\
			& \lesssim b^{-\varepsilon n_0} + C\mu(\mathcal{B}_{n_0}^\ast) \\
			& \lesssim b^{-\varepsilon n_0} + C b^{-c_1 n_0}\\
			& \lesssim b^{-c_1 n_0} < C\frac{n_0}{n_1} 
		\end{align*}
		
		Having 
		\begin{equation*}
			\left|\ell_{n_1}(g_1)  +\frac{n-2}{n}\ell_{n_0}(g_1) - \frac{2(n-1)}{n}\ell_{2n_0}(g_1)\right| < C\frac{n_0}{n_1}
		\end{equation*}
		one may write
		\begin{equation*}
			\left|\ell_{n_1}(g_1) + \ell_{n_0}(g_1) -  2\ell_{2n_0}(g_1) - \frac{2}{n}\left[\ell_{n_0}(g_1)-\ell_{2n_0}(g_1)\right]\right|<C\frac{n_0}{n_1}
		\end{equation*}
		so that (\ref{inductiveCont1}) holds:
		\begin{equation*}
			\left|\ell_{n_1}(g_1) + \ell_{n_0}(g_1) -  2\ell_{2n_0}(g_1) \right| < C\frac{n_0}{n_1}.
		\end{equation*}
		The same process may be used to obtain (\ref{inductiveCont1}) at times $2n_1$. Then by an immediate triangle inequality one obtains (\ref{inductiveCont2}).
		
		To prove (\ref{inductiveCont3}) start by rewriting (\ref{inductiveCont1}) as
		\begin{equation*}
			|\ell_{n_1}(g_1)-\ell_{n_0}(g_1) + 2[\ell_{n_0}(g_1)-\ell_{2n_0}(g_1)] | <C\frac{n_0}{n_1}.
		\end{equation*}
		So
		\begin{align*}
			|\ell_{n_1}(g_1)-\ell_{n_1}(g)| & \leq |\ell_{n_1}(g_1)-\ell_{n_0}(g_1) + 2[\ell_{n_0}(g)-\ell_{2n_0}(g)] | \\
			& \hspace{1cm} + |\ell_{n_1}(g)-\ell_{n_0}(g) + 2[\ell_{n_0}(g)-\ell_{2n_0}(g)] | \\
			& \hspace{1cm} + 2|\ell_{n_0}(g_1)-\ell_{2n_0}(g_1)| +2|\ell_{n_0}(g)-\ell_{2n_0}(g)| \\
			& \hspace{1cm} + |\ell_{n_0}(g)-\ell_{n_0}(g_1)| \\
			& < \theta_0 + 4\eta_0 + C\frac{n_0}{n_1} =:\theta_1.
		\end{align*}
	\end{proof}
	
	\subsection{Rate of convergence}
	
	In this section we shall use the inductive step to understand exactly how pushing  the controls though the natural numbers affects the convergence rate of the quantities at hand.
	
	\begin{lemma}
		Let $\{x_n\}$ be a sequence converging to $x$ such that for every $n\in \mathbb{N}$,
		\begin{equation*}
			|x_{n}-x_{2n}|<\frac{\log_b n }{n},
		\end{equation*}
		then, for every $n\in \mathbb{N}$
		\begin{equation*}
			|x_n - x| \lesssim \frac{\log_b n }{n}.
		\end{equation*}
	\end{lemma}
	
	\begin{proof}
		Let $n\in \mathbb{N}$, then we can use a telescopic sum to write
		\begin{align*}
			|x_n-x| & = \left| \sum_{i=0}^{\infty} x_{2^in}-x_{2^{i+1}n} \right| \leq  \sum_{i=0}^{\infty} \left| x_{2^in}-x_{2^{i+1}n} \right|\\
			& \leq  \sum_{i=0}^{\infty}  \frac{\log_b(2^in)}{2^in} \lesssim \frac{\log_b n }{n},
		\end{align*}
		as the sum of the series is of order $\frac{\log_b n}{n}$. 
	\end{proof}
	
	Using the inductive step we can now obtain the rate of  convergence associated with the functions $\ell_n$. These however will be too slow, hence we also look at $-\ell_{n} + 2\ell_{2n}$.
	
	\begin{proposition}
		\label{convergeRate}
		Let $g\in \mathcal{C}$. There are constants $r_1>0$, $\overbar{n_0}\in \mathbb{N}$, $c_2>0$, $K<\infty$ all depending on $g$ such that the following hold
		\begin{align*}
			|\ell(g_1)-\ell_{n}(g_1)| & <K\frac{\log_b n }{n}\\
			|\ell(g_1) + \ell_{n}(g_1) - 2\ell_{2n}(g_1)|&< b^{-c_2n},
		\end{align*}
		for every $n>\overbar{n_0}$ and $g_1\in S^\infty(\Omega, G)$ with $d_\infty(g, g_1)<r_1$. 
	\end{proposition}
	
	\begin{proof}
		Let us use the constants $\varepsilon, \, c_1, \, C, \, r$ and $\overbar{n_0}$ given in the inductive step.  Consider the quantities $n_0^- = \overbar{n_0}$,  $n_0^+ = \ceil{b^{c_1\overbar{n_0}}} $ and set $\mathcal{N}_0 := [n_0^-, n_0^+]$. We shall also define $r_1 = \min \{r, b^{-3C_1\overbar{n_0}} \}$. Then, by the finite scale continuity, for every $n_0\in \mathcal{N}_0$, we have
		\begin{equation*}
			|\ell_{2n_0}(g_1)-\ell_{2n_0}(g)|<\frac{ b^{2C_1n_0}}{\log b}d_\infty(g_1,g_2) \leq b^{-C_1\overbar{n_0}} \leq \varepsilon,
		\end{equation*}
		choosing $\overbar{n_0}$ large enough for the effect. Likewise 
		\begin{equation*}
			|\ell_{n_0}(g_1)-\ell_{n_0}(g)|< \varepsilon =: \theta_0,
		\end{equation*}
		Moreover 
		\begin{equation*}
			|\ell_{2n_0}(g)-\ell_{n_0}(g)|< |\ell_{2n_0}(g)-\ell(g)| + |\ell(g)-\ell_ {n_0}(g)| < 2\varepsilon ,
		\end{equation*}
		so that
		\begin{equation*}
			|\ell_{2n_0}(g)-\ell_{n_0}(g)|< 2\varepsilon =: \eta_0,
		\end{equation*}
		and we have
		\begin{equation*}
			\theta_0 + 2 \eta_0 = 5\varepsilon< \ell(g)-6\varepsilon.
		\end{equation*}
		
		Using the inductive process we now have $n_1^- = \floor{b^{c_1n_0^-}}$, $n_1^+ = \ceil{b^{c_1n_0^+}}$ and define $\mathcal{N}_1 = [n_1^-, n_1^+]$. If $n_1\in \mathcal{N}_1$ then $n_0 \lesssim \log_b(n_1)$. Now,
		\begin{equation*}
			\left|\ell_{n_1}(g_1) + \ell_{n_0}(g_1) -  2\ell_{2n_0}(g_1) \right| < C\frac{n_0}{n_1} < K\frac{\log_b n_1}{n_1},
		\end{equation*}
		for some constant $K$. Moreover
		\begin{align*}
			\ell_{n_1}(g_1) - \ell_{2n_1}(g_1) & < \eta_1 \\
			|\ell_{n_1}(g_1) - \ell_{n_1}(g)|& < \theta_1
		\end{align*}
		where
		\begin{align*}
			\theta_1 & = \theta_0 + 4\eta_0 + C\frac{n_0}{n_1} < 13 \varepsilon + K\frac{\log_b n_1 }{n_1}, \\
			\eta_1 & = C\frac{n_0}{n_1} < K\frac{\log_b n_1}{n_1}. 
		\end{align*}
		
		Furthermore,
		\begin{equation*}
			\theta_1 + 2\eta_1 \leq 13 \varepsilon + 3K\frac{\log_b n_1 }{n_1} < 16\varepsilon < \ell(g) -6\varepsilon. 
		\end{equation*}
		Hence we can repeat the process, let $n_2^-=\floor{b^{c_1n_1^-}}, n_2^+=\ceil{b^{c_1n_1^+}}$, and define 
		$\mathcal{N}_2=[n_2^-, n_2^+]$, then, if $n_2 \in \mathcal{N}_2$, there exists $n_1 \in \mathcal{N}_1$ such that  $n_1 \lesssim \log_b(n_2)$
		\begin{equation*}
			\left|\ell_{n_2}(g_1) + \ell_{n_1}(g_1) -  \ell_{2n_1}(g_1) \right| < C\frac{n_1}{n_2} < K\frac{\log_b n_2 }{n_2}.
		\end{equation*}
		Moreover
		\begin{align*}
			\ell_{n_2}(g_1) - \ell_{2n_2}(g_1) & < \eta_2 \\
			|\ell_{n_2}(g_1) - \ell_{n_2}(g)|& < \theta_2
		\end{align*}
		where
		\begin{align*}
			\theta_2 & = \theta_1 + 4\eta_1 + C\frac{n_1}{n_2} < 13 \varepsilon + 5 K\frac{\log_b n_1 }{n_1} + K\frac{\log_b n_2 }{n_2}, \\
			\eta_2 &  = C\frac{n_1}{n_2}<K\frac{\log_b n_2 }{n_2}
		\end{align*}
		
		Inductively repeating the process we obtain intervals $\mathcal{N}_k$ whose union cover all natural numbers greater than $n_0$. Hence given $n>n_0$, there exists $k\geq0$ such that $n\in \mathcal{N}_{k+1}$, so there is also $n_k\in \mathcal{N}_k$ so that
		\begin{equation*}
			n= n_{k+1} = \ceil{b^{c_1n_k}}.
		\end{equation*}
		Moreover
		\begin{equation*}
			\ell_{n_{k+1}}(g_1)-\ell_{2n_{k+1}}(g_1) < \eta_{k+1} < K\frac{\log_b n_{k+1} }{n_{k+1}}
		\end{equation*}
		and 
		\begin{align*}
			|\ell_{n_{k+1}}(g_1)-\ell_{n_{k+1}}(g_1)| &< \theta_{k+1} \\
			& < \theta_k + 4 \eta_{k+1} + C\frac{n_k}{n_{k+1}} \\
			& < 13 \varepsilon + 5K\sum_{i=1}^{k}K\frac{\log_b n_i }{n_i} + K\frac{\log_b n_{k+1}}{n_{k+1}},
		\end{align*}
		however, since $n_k$ increase super-exponentially, the series $\sum_{i>0} \frac{\log_b n_i }{n_i}$ is convergent with sum of order $\frac{\log_b n_1 }{n_1}$.

		With that, for every $n\geq n_0$ we obtain
		\begin{equation*}
			\ell_n(g_1)-\ell_{2n}(g_1) < K \frac{\log_b n }{n} 
		\end{equation*}
		whence
		\begin{equation*}
			|\ell_n(g_1)-\ell(g_1)| < K \frac{\log_b n }{n}.
		\end{equation*}
		
		Now, 
		\begin{align*}
			\left|\ell_{n_{k+1}}(g_1) + \ell_{n_k}(g_1) -  2\ell_{2n_k}(g_1) \right| & <  K\frac{\log_b n_{k+1} }{n_{k+1}} \\
			& \leq K c_1n_k b^{-c_1 n_k} <  b^{-\frac{c_1}{2}n}
		\end{align*}
		so
		\begin{equation*}
			\left|\ell(g_1) + \ell_{n_k}(g_1) -  2\ell_{2n_k}(g_1) \right|< 2b^{-\frac{c_1}{2}n_k}< b^{-\frac{c_1}{3}n_k}
		\end{equation*}
		hence the result follows for $n>n_0$.
	\end{proof}

	\subsection{Proof of items 1) and 2) in Theorem \ref{ContTheorem}}
	
	Recall from Lemma \ref{USC}, since $d(gx_0, x_0)\geq 0$ we  already have continuity for cocycles in $\mathcal{C}$ with zero drift, so it remains to obtain the part regarding $\mathcal{C}_+$, so point 1) follows from point 2), which we now prove.
	
	\begin{proof}
		Consider $\overbar{n_1}\in \mathbb{N}$, $r_1>0$, $r>0$, $c_2$ as in Proposition \ref{convergeRate} and $C_1$ and is Proposition \ref{FiniteScaleContinuity}. Let $g\in \mathcal{C}_+$ with $\ell(g)>0$ and take the function $f_n:\mathcal{C} \to \mathbb{R}$
		\begin{equation*}
			f_n := -\ell_n + 2\ell_{2n}
		\end{equation*}
		clearly $f_n(g) \to \ell(g)$, moreover an exponential rate of convergence holds for every $n \geq \overbar{n_0}$,
		\begin{equation*}
			|\ell(g_1)-f_n(g_1)| = |\ell(g_1)+\ell_n(g_1)-2\ell_{2n}(g_1)| \leq b^{-c_2n}.
		\end{equation*}
		Consider now $d_\infty(g_1, g_2)<\log(b)b^{-2( C_1 + c_2)\overbar{n_1}}$, and pick $n\geq \overbar{n_1}$ such that
		\begin{equation*}
			b^{-4( C_1 + c_2)n} < d_\infty(g_1\, , \, g_2) < b^{-2( C_1 + c_2)n}.
		\end{equation*}
		Then for $m$ equal to either $n$ or $2n$ one has
		\begin{equation*}
			|\ell_m(g_1)-\ell_m(g_2)|\leq \frac{b^{ 2C_1n}}{\log b } d_\infty(g_1, g_2) < b^{-2c_2n}
		\end{equation*}
		thus
		\begin{align*}
			|f_n(g_1)-f_n(g_2)| &\leq |\ell_{n}(g_1)-\ell_{n}(g_2)| + 2|\ell_{2n}(g_1)-\ell_{2n}(g_2)| \\
			& \leq 3 b^{-2c_2n} \leq b^{-c_2n}.
		\end{align*}
		Finally one has
		\begin{align*}
			|\ell(g_1) - \ell(g_2)| & \leq |\ell(g_1)-f_n(g_1)| + |f_n(g_1) - f_n(g_2)|  + |\ell(g_2)-f_n(g_2)| \\
			& \leq 3b^{-c_2 n} \\
			& \leq 3 d_\infty(g_1\, , \, g_2)^\alpha,
		\end{align*}
		where $\alpha = \frac{c_2}{4( C_1 + c_2)}$.
	\end{proof}
	
	\subsection{Large deviations remark}
	
	Given $g\in S_{+}^\infty(\Omega, G)$, by the rate  of convergence, there exists a neighbourhood $V$ of $g$ in  $S_{+}(\Omega, G)$ and $\overbar{n_1}\in \mathbb{N}$ such that the finite scale drifts $\ell_n$ converge uniformly to $\ell$ on $V$. Hence, for every $\varepsilon>0$ there exists $\overbar{n}(\varepsilon)$ such that for every $n\geq \overbar{n}(\varepsilon)$ and $g_1 \in V$,
	\begin{align*}
		|\ell(g_1)-\ell(g)| & <\varepsilon \\
		|\ell_n(g_1)-\ell(g_1)| & <\varepsilon.
	\end{align*}
	Therefore large deviation estimates can be restated in a stronger manner
	
	\begin{definition}[Uniform large deviation estimates]
		Given $g\in S_{+}^\infty(\Omega, G)$ There exists a neighbourhood $V\subset S^\infty(\Omega, G)$ of $g$ and a constant $c>0$ such that for every $\varepsilon>0$, there exists $\overbar{n}_0$ such that 
		\begin{equation*}
			\mu\left\{ \omega \in \Omega \, : \, \left| \frac{1}{n}d(g_1^{(n)}(\omega)x_0\, , \, x_0) -  \ell(g_1) \right|> \varepsilon \right\} < b^{-cn\varepsilon^2},
		\end{equation*}
		for every $g_1 \in V$ and $n \geq \overbar{n}_0$.
	\end{definition}

	\section{Continuity of the hitting point}
	
	Proving the continuity of the tracking point is similar to proving the continuity of the drift although some of the hard work has already been done. 
	
	Let $g\in S^\infty(\Omega, G)$, we start by considering the positional maps
	\begin{align*}
		p_g^{(n)} : \Omega &\to X \\
		\omega & \mapsto g^{(n)}(\omega)x_0
	\end{align*}
	and consider their limit in $\bord X$
	\begin{equation*}
		p_g^{(\infty)}(\omega) := \lim_{n\to \infty} p_g^{(n)}(\omega),
	\end{equation*}
	whose existence we shall discuss later in section 4.2.  Notice that if $g\in S_+^\infty(\Omega, G)$, then for almost every $\omega \in \Omega$
	\begin{equation*}
		\xi_g(\omega) = p_g^{(\infty)}(\omega).
	\end{equation*} 
	
	Given $g_1, g_2\in S^\infty(\Omega, G)$ we define the quantity
	\begin{equation}
		d_1(p_{g_1}, p_{g_2}) = \int_\Omega \overbar{D}_b(g_1(\omega)x_0, g_2(\omega)x_0) d\mu(\omega).
	\end{equation}
	The route to prove continuity of $\xi$ is the same as the one done before for the drift $\ell$. We check the finite scale continuity with respect to $d_1$ first and then we compute the rate of convergence. Since the space is strongly hyperbolic we then obtain 
	\begin{equation*}
		d_1(\xi_{g_1}(\omega), \xi_{g_2}(\omega)) = \lim_{n\to \infty } d_1(p_{g_1}^{(n)}, p_{g_2}^{(n)}).
	\end{equation*}

	\subsection{Finite scale continuity}

	\begin{proposition}
		Let $g\in S_+^\infty(\Omega, G)$, there exist $c=c(g)>0$,  $r>0$, $\varepsilon>0$ and $C_2=C_2(g,\varepsilon)<\infty$ such that for every $g_1, g_2 \in S^\infty(\Omega, G)$ with $d_\infty(g, g_i)<r$ if $n\geq \overbar{n}(\varepsilon)$ and $d_\infty(g_1, g_2)< b^{-C_2n}$, then for every $\omega$ outside a set of measure $< b^{-nc\varepsilon
			^2}$
		\begin{equation*}
			\overbar{D}_b(p_{g_1}^{(n)}(\omega) \, , \, p_{g_2}^{(n)}(\omega) ) \leq b^{-nc\varepsilon^2}.
		\end{equation*}
		Hence
		\begin{equation*}
			d_1(p_{g_1}^{(n)} \, , \, p_{g_2}^{(n)} ) \lesssim b^{-nc\varepsilon^2}.
		\end{equation*}
	\end{proposition}
	
	\begin{proof}
		Consider $c$ to be the large deviation parameter. By the continuity of $\ell(g)$, take $0<\gamma_1< \ell(g) < \gamma_2$ close enough so that
		\begin{align*}
			\gamma_1 & < \inf\{ \ell(g_*) \, : \, g_* \in S^\infty(\Omega, G) \textrm{ and } d_\infty(g, g_*)<r \} \\
			& \leq \sup\{ \ell(g_*) \, : \, g_* \in S^\infty(\Omega, G) \textrm{ and } d_\infty(g, g_*)<r \}  < \gamma_2,
		\end{align*}
		as well as $\varepsilon>0$ so that 
		\begin{equation*}
			c\varepsilon^2 \leq \gamma_1 \leq \ell(g_*) - \varepsilon \leq \ell(g_*) + \varepsilon \leq \gamma_2
		\end{equation*}
		for every $g_* \in S^\infty(\Omega, G)$ and $d_\infty(g, g_*)<r$.
		
		For every $n\geq \overbar{n}(\varepsilon)$, the deviation sets 
		\begin{equation*}
			\mathcal{B}_n(g_*) = \left\{\omega \in \Omega \, : \, \left| \frac{1}{n}d(g_*^{(n)}(\omega)x_0\, , \, x_0) - \ell(g_*) \right| > \varepsilon \right\}
		\end{equation*}
		have their measure bounded by $\lesssim b^{-nc\varepsilon^2}$.
		
		Let $\omega\notin \mathcal{B}_n(g_1)\cup \mathcal{B}_n(g_2)$, then for $i=1,2$
		\begin{align*}
			d(g_i^{(n)}(\omega)x_0 , x_0) &< n(\ell(g_i) + \varepsilon) < n\gamma_2, \\
			d(g_i^{(n)}(\omega)x_0 , x_0)& > n(\ell(g_i) - \varepsilon) > n\gamma_1.
		\end{align*}
		
		At this point, notice as in the proof of  Proposition \ref{FiniteScaleContinuity}
		\begin{align*}
			d(g_1^{(n)}(\omega)x_0\, , \, g_2^{(n)}(\omega)x_0) & \leq \frac{1}{\log b}b^{d(g_1^{(n)}(\omega)x_0\, , \, g_2^{(n)}(\omega)x_0)/2} \\
			& \leq \frac{b^{n\gamma_2}}{\log b}d_\infty(g_1, g_2).
		\end{align*}
		
		Finally, choosing $C_2 > \gamma_2$, provided $d_\infty(g_1,g_2)< b^{C_2n}$,
		\begin{align*}
			\overbar{D}_b(p_{g_1}^{(n)}(\omega) \, , \, p_{g_2}^{(n)}(\omega) ) & \leq b^{-\langle g_1^{(n)}(\omega)x_0\, , \, g_2^{(n)}(\omega)x_0 \rangle_{x_0}} \\
			& \leq b^{\frac{1}{2}\left[d(g_1^{(n)}(\omega)x_0, g_2^{(n)}(\omega)x_0) - d(g_1^{(n)}(\omega)x_0, x_0) - d(g_2^{(n)}(\omega)x_0, x_0)\right]}\\
			& \leq b^{-n\gamma_1}  \leq b^{-nc\varepsilon^2}.
		\end{align*}
	\end{proof}
	
	\subsection{Rate of Convergence}
	
	\begin{proposition}
		Let $g\in S_+^\infty(\Omega, G)$, then $(p_{g_1}^{(n)})$ is a Cauchy sequence, in particular $p_{g_1}^{(\infty)}$ is well defined. Moreover, there are constants $r>0$, $\varepsilon>0$ and $\overbar{n_0}\in \mathbb{N}$, all depending on $g$, such that
		\begin{equation*}
			d_1(p_{g_1}^{(n)}\, , \, p_{g_1}^{(\infty)}) \lesssim b^{-nc\varepsilon^2}
		\end{equation*}
		for all  $n \geq \overbar{n_0}$ and for all $g_1\in S^\infty(\Omega, G)$ with $d_\infty(g_1, g)<r$
		
	\end{proposition}
	
	\begin{proof}
		Consider $\gamma_1$ and $\varepsilon$ given as in the proof of the previous proposition and $c$ the large deviation paramenter. As well as the deviation sets
		\begin{equation*}
			\mathcal{B}_n(g_1) = \left\{\omega \in \Omega \, : \, \left| \frac{1}{n}d(g_1^{(n)}(\omega)x_0\, , \, x_0) - \ell(g_1) \right| > \varepsilon \right\}
		\end{equation*}
		Recall the control, $d_\infty(g_1)<C$(see Lemma \ref{normIneq}) for every $\omega \notin \mathcal{B}_n(g_1)$ 
		\begin{align*}
			\overbar{D}_b(g_1^{(n)}(\omega) x_0\, , \, g_1^{(n+1)}(\omega)x_0) & \leq b^{\frac{1}{2} \left[  d(g_1(T^n\omega)x_0, x_0) - d(g_1^{(n)}(\omega)x_0,x_0)- d(g_1^{(n+1)}(\omega)x_0,x_0)\right]} \\
			& \leq \sqrt{C} b^{-n\gamma_1}
		\end{align*}
		Hence, for every $m>n$ 
		\begin{align*}
			\overbar{D}_b(g_1^{(n)}(\omega) x_0\, , \, g_1^{(m)}(\omega)x_0) & \leq \sum_{i=n}^{m-1} D_b(g_1^{(i)}(\omega) x_0\, , \, g_1^{(i+1)}(\omega)x_0) \\
			& \leq \sqrt{C} \sum_{i=n}^{m-1}  b^{-i\gamma_1} \\
			& \leq \frac{\sqrt{C}}{1-b^{-\gamma_1
			}} b^{-n\gamma_1},
		\end{align*}
		hence $g_1^{(n)}(\omega)x_0$ is a Gromov sequence, in particular it converges to some point in $\partial X$. With this we obtain $\overbar{D}_b(p_{g_1}^{(n)}\, , \, p_{g_1}^{(\infty)})\lesssim b^{-nc\varepsilon^2}$. Integrating over $\omega$ yiels the result.
	\end{proof}
	
	The proof of item 3) in Theorem \ref{ContTheorem} is now analogue to that of item 1).
	
	\section{Large Deviations for the Drift in Markov Systems}
	
	In this section we obtain the large deviations. Although the method used is based in Nagaev's \cite{nagaev1957some}, we will apply Duarte and Klein's recipe \cite{duarte2016lyapunov}. In §\ref{method} we describe the recipe and ready the ingredients laid by Duarte and Klein whilst §\ref{methodApp} is devoted to proving the large deviations. Many of the arguments displayed here are an adaptation of what was  done in $\cite{sampaio2021regularity}$ for random walks.

	Let us recall the reader once more that $X$ stands for a $\delta$-hyperbolic metric space with a basepoint $x_0$, $G$ for its groups of isometries and $b$ for a real number between $1$ and $2^{1/\delta}$.
	
	\subsection{The method}
	\label{method}
	
	Consider a Markov system $(K, \mu)$ on a metric space $\Gamma$ and let $\Omega = \Gamma^\mathbb{N}$. Given some Borel measurable observable $\zeta:\Gamma \to \mathbb{R}$, let $\hat{\zeta}:\Omega \to \mathbb{R}$ be the Borel measurable function $\hat{\zeta}(\omega) = \zeta(\omega_0)$. We call a sum process of $\zeta:\Gamma \to \mathbb{R}$ the sequence of random variables $\{S_n(\zeta)\}$ on $(\Omega, \mathcal{B})$,
	\begin{equation*}
		S_n(\zeta)(\omega) := \sum_{i=0}^{n-1} \hat{\zeta} \circ T^j(\omega) = \sum_{i=0}^{n-1} \zeta(\omega_j).
	\end{equation*}
	An observed Markov system on $\Gamma$ is a triple $(K, \mu, \zeta)$ where $(K, \mu)$ is a Markov system on $\Gamma$ and $\zeta: \Gamma\to \mathbb{R}$ is a Borel-measurable function.
	
	Recall that $\mathbb{P}_{\omega_0}$ stands for the Kolmogorov extension of  $(K,\delta_{\omega_0})$.
	\begin{definition}[Large deviations of exponential type]
		We say that $\zeta$ satisfies large deviation of exponential type if there exist positive constants $b, C, k, \varepsilon_0$ and $n_0$ such that for all $n>n_0$, $0<\varepsilon<\varepsilon_0$ and $\omega_0\in \Gamma$,
		\begin{equation*}
			\mathbb{P}_{\omega_0} \left\{ \omega\in \Omega \, : \, \left|\frac{1}{n}S_n(\zeta)(\omega)-\mathbb{E}_\mu(\zeta)\right| > \varepsilon \right\} \leq Cb^{-k\varepsilon^2n}.
		\end{equation*}
	\end{definition}
	
	We will obtain the large deviations in Theorem \ref{LDT} by exploring the properties of contracting operators on suitable Banach spaces. Let us start by introducing the operators. Consider $K$ a Markov kernel on a metric space $\Gamma$, the operator $Q_K : L^\infty(\Gamma) \to L^\infty(\Gamma)$, given by
	\begin{equation*}
		(Q_Kf)(\omega_0) = \int_\Gamma f(\omega_1)K(\omega_0, d\omega_1),
	\end{equation*}
	is called the Markov operator. The Markov operator allows us to characterize stationary measure in a more useful way, with effect, $\mu$ is $K$-stationary if and only if
	\begin{equation*}
	    \int Q_Kf d\mu = \int f d\mu
	\end{equation*}
	for every $f\in L^1(\Gamma)$.
	Let now $(K, \mu, \zeta)$ be an observed Markov system  on a given metric space $\Gamma$, then we call the operator $Q_{K,\zeta}:L^\infty(\Gamma)\to L^\infty(\Gamma)$ given by
	\begin{equation*}
		(Q_{K,\zeta}f)(\omega_0) := \int_\Gamma f(\omega_1)b^{\zeta(\omega_1)} K(\omega_0, d\omega_1),
	\end{equation*}
	and $b>0$ the Laplace-Markov operator.
	
	We will now follow closely \cite{duarte2016lyapunov} as we introduce a series of assumptions, eleven to be exact, which yield an abstract LDT. In the next section we make sense of this setting and prove the assumptions hold as to obtain the large deviations. The main difference between the two settings is the fact that we apply these results to not necessarily compact spaces.
	
	Let $(\mathcal{M}, \textrm{dist})$ be a metric space of observed Markov systems $(K, \mu, \zeta)$  on a given metric space $\Gamma$. Consider as well a scale of Banach algebras $(B_\alpha, ||\cdot ||_\alpha)$ indexed in $\alpha \in [0,1]$, where each $B_\alpha$ is a space of bounded Borel measurable functions on $\Gamma$. We assume that there exists seminorms $\upsilon_\alpha:B_\alpha \to [0, +\infty)$ such that for every $0\leq \alpha \leq 1$,
	\begin{enumerate}
		\item [A1)] $||f||_\alpha = \upsilon_\alpha(f) + ||f||_\infty$,
		\item [A2)] $B_0=L^\infty(\Sigma)$ and $||\cdot||_0$ is equivalent to $||\cdot||_\infty$,
		\item [A3)] $B_\alpha$ is a lattice, i.e., if $f\in B_\alpha$ then $\overbar{f}, |f|\in B_\alpha$,
		\item [A4)] $B_\alpha$ is a Banach algebra with unity $\mathbf{1}\in B_\alpha$ and $\upsilon_\alpha(\mathbf{1})=0$.
	\end{enumerate}
	Assume also that for every $0\leq\alpha_0 < \alpha_1 < \alpha_2 \leq1$,
	\begin{enumerate}
		\item [B1)] $B_{\alpha_2}\subset B_{\alpha_1}\subset B_{\alpha_0} $,
		\item [B2)] $\upsilon_{\alpha_0}(f) \leq \upsilon_{\alpha_0}(f) \leq \upsilon_{\alpha_0}(f)$, for every $f\in B_{\alpha_2}$,
		\item [B3)] $\upsilon_{\alpha_1}(f) \leq \upsilon_{\alpha_0}(f)^{\frac{\alpha_2-\alpha_1}{\alpha_2-\alpha_0}} \, \upsilon_{\alpha_2}(f)^{\frac{\alpha_1-\alpha_0}{\alpha_2-\alpha_0}}$, for every $f\in B_{\alpha_2}$.
	\end{enumerate}
	The assuptions $A*)$ and $B*)$ exhaust our assuptions on the Banach algebras and will be the simple part of what is to come. Finally, for our assumptions on $\mathcal{M}$, assume there exists an interval $[\alpha_1, \alpha_0]\subset(0,1]$ with $\alpha_1 < \alpha_0/2$ such that  for every $\alpha \in [\alpha_1, \alpha_0]$ the following properties hold,
	\begin{enumerate}
		\item [C1)] $ (K,\mu, -\zeta)\in \mathcal{M}$, whenever $ (K,\mu, \zeta)\in \mathcal{M}$.
		\item [C2)] The Markov operators $Q_K:B_\alpha \to B_\alpha$ are uniformly strongly mixing. That is, there exist $C>0$ and $0<\sigma<1$ such that for every $(K, \mu, \zeta)\in \mathcal{M}$ and $f\in B_\alpha$,
		\begin{equation*}
			\left|\left|Q_K^nf - \int_{\Sigma}f(\omega_0) d\mu (\omega_0)\right|\right|_\alpha \leq C\sigma^n ||f||_ \alpha.	
		\end{equation*}
		\item [C3)] The operators $Q_{K, z\zeta}$ act continuously on the Banach algebra $B_\alpha$ uniformly in $ (K,\mu, \zeta)\in \mathcal{M}$. With effect, we assume, there are positive constants $c$ and $M$ such that for $i=0,1,2$, $|z|<c$ and $f\in B_\alpha$
		\begin{equation*}
			Q_{K, z\zeta}(f\zeta^i) \in B_\alpha \, \textrm{  and  } \, ||Q_{K, z\zeta}(f\zeta^i)|| \leq M ||f||_\alpha.
		\end{equation*}
		\item [C4)] Consider the family of maps $(K,\mu, \zeta) \to Q_{K, z\zeta}$ indexed in $|z|<c$, there exists $0<\theta\leq 1$ such that for every $|z|<b$, $f\in B_\alpha$ and $  (K_1,\mu_1, \zeta_1), (K_2,\mu_2, \zeta_2)\in \mathcal{M}$,
		\begin{equation*}
			|| Q_{K_1, z\zeta_1}f- Q_{K_2, z\zeta_2}f||_\infty \leq M||f||_\alpha \textrm{dist}\left((K_1,\mu_1, \zeta_1), (K_2,\mu_2, \zeta_2)\right)^\theta.
		\end{equation*}
	\end{enumerate}
	
	Under all these assumptions the following abstract LDT theorem holds:
	
	\begin{theorem}[in \cite{duarte2016lyapunov}]
		\label{LDTT}
		Given $(K_0,\mu_0, \zeta_0)\in \mathcal{M}$ and $0<s<\infty$ large enough (which can be made precise), there exists a neighbourhood $V$ of $(K_0,\mu_0, \zeta_0)\in \mathcal{M}$, $C>0$, $\varepsilon_0>0$ and $n_0 \in \mathbb{N}$, such that for every $(K,\mu, \zeta)\in V$, $0< \varepsilon < \varepsilon_0$, $\omega_0\in \Sigma$ and $n>n_0$
		\begin{equation*}
			\mathbb{P}_{\omega_0} \left[ \left| \frac{1}{n}S_n(\zeta) - \mathbb{E}_ \mu(\zeta) \right| \geq \varepsilon \right] \leq Cb^{-\frac{\varepsilon^2}{s}n}.
		\end{equation*} 
		which averaging over $\omega_0$ with respect to $\mu$ yields 
		\begin{equation*}
			\mathbb{P}_{\mu} \left[ \left| \frac{1}{n}S_n(\zeta) - \mathbb{E}_ \mu(\zeta) \right| \geq \varepsilon \right] \leq Cb^{-\frac{\varepsilon^2}{s}n}.
		\end{equation*} 
	\end{theorem}
	
	\begin{remark}
	    By choosing a large $s$ and $n\geq \bar{n}(\varepsilon)$ we can make $C=1$, thus obtaining large deviations as in Definition \ref{LDE}.
	\end{remark}
	
	\subsection{Obtaining the Large deviations}
	\label{methodApp}
	
	Let $X$ be an Hyperbolic metric space, $X^h$, $\partial X$ denote its horofunctions compactification and Gromov boundary respectively. We denote by $D_b$ the visual metric on $\partial X$, where $1<b\leq2^{1 / \delta}$ is fixed. In this section we use Theorem $\ref{LDTT}$ to obtain our large deviations for the drift. From this point on $\Sigma$ stands for a compact metric space; and in the scope of the previous section $\Gamma$ stands for $\Sigma \times \Sigma \times \partial X$ with the product metric.
	
	\subsubsection{Verifying conditions A*) and B*)}
	
	Given $0 \leq \alpha \leq 1$ and $f\in L^\infty(\Gamma)$, define
	\begin{align*}
		\upsilon_\alpha(f) & :=  \sup_{\substack{(\omega_1,\omega_2)\in \Sigma \times \Sigma \\ \xi \neq \eta} } \frac{|f(\omega_1,\omega_2, \xi)-f(\omega_1,\omega_2,\eta)|}{D_b(\xi, \eta)^\alpha},\\
		||f||_\alpha & := ||f||_\infty + \upsilon_\alpha(f),
	\end{align*}
	and set 
	\begin{equation*}
		\mathcal{H}_\alpha(\Gamma) : = \left\{\ f\in L^\infty(\Gamma) \, : \, ||f||_\alpha < \infty \right\}.
	\end{equation*}
	the space of boundary Hölder continuous functions in $\Gamma$. We call $\upsilon_\alpha(f)$ the boundary Hölder exponent of $f$.
	
	\begin{proposition}
		The family $\{\mathcal{H}_\alpha(\Gamma)\}$, for $0\leq \alpha \leq 1$, consist of Banach algebras with norm $||f||_\alpha$ satisfying the conditions $A*)$ and $B*)$.
	\end{proposition}
	
	\begin{proof}
		
		It is a standard proof that $\mathcal{H}_\alpha(\Gamma)$ are Banach algebras. Now points $A1)$, $A3)$, $B1)$, $B2)$ are either clear or follow from some immediate computation. For point $A2)$ notice that for $\alpha =0$ we have $||f||_\alpha \leq 2||f||_\infty$. Point $A4)$ follows from the immediate inequality
		\begin{equation*}
			\upsilon_\alpha(fg) \leq ||f||_\infty \upsilon_\alpha(g) + ||g||_\infty \upsilon_\alpha(f).
		\end{equation*}
		For point $B3)$, notice that given $\alpha_0, \alpha_2, s \in [0,1]$,
		\begin{align*}
			\upsilon_{s\alpha_0 +(1-s)\alpha_2}(f) &  = \sup_{\substack{(\omega_1,\omega_2)\in \Sigma \times \Sigma \\ \xi \neq \eta} } \frac{|f(\omega_1,\omega_2,\xi)-f(\omega_1,\omega_2,\eta)|^{s+(1-s)}}{D_b(\xi, \eta)^{s\alpha_0 +(1-s)\alpha_2}} \\
			& \leq \sup_{\substack{(\omega_1,\omega_2)\in \Sigma \times \Sigma \\ \xi \neq \eta} } \frac{|f(\omega_1,\omega_2,\xi)-f(\omega_1,\omega_2,\eta)|^{s}}{D_b(\xi, \eta)^{s\alpha_0}} \\
			&  \hspace{2cm} \times
			\sup_{\substack{(\omega_1,\omega_2)\in \Sigma \times \Sigma  \\ \xi \neq \eta} } \frac{|f(\omega_1,\omega_2,\xi)-f(\omega_1,\omega_2,\eta)|^{(1-s)}}{D_b(\xi, \eta)^{(1-s)\alpha_2}} \\ \\
			& =  \upsilon_{\alpha_0}(f)^s \upsilon_{\alpha_2}(f)^{1-s},
		\end{align*}
		picking $s= \frac{\alpha_2-\alpha_1}{\alpha_2-\alpha_0}$ the result follows.
	\end{proof}

	\subsubsection{Verifying conditions C*)}
	
	Recall the space $S^\infty(\Sigma \times \Sigma, G)$ of bounded measurable cocycles $g:\Sigma \times \Sigma \to G$ introduced in section 1.3. Each cocycle $g\in S^\infty(\Sigma \times \Sigma,\, G)$ defines a Markov kernel on $\Gamma$ given by
	\begin{equation*}
		\overbar{K}_g(\omega_0, \omega_1, \xi) := \int_\Sigma \delta_{(\omega_1, \omega_2, g(\omega_1, \omega_2)^{-1} \xi)} K(\omega_1, d\omega_2),
	\end{equation*}
	as well as an associated Markov operator $Q_g:L^\infty(\Gamma) \to L^\infty(\Gamma)$ with expression
	\begin{equation*}
		(\overbar{Q}_gf)(\omega_0,\omega_1, \xi) := \int_\Sigma f(\omega_1,\, \omega_2, \, g(\omega_1, \omega_2)^{-1} \xi) K(\omega_1, d\omega_2).
	\end{equation*}
	The reason for looking at the action of the inverse comes from the relation (\ref{horoAction}). For each $g\in C(\Sigma \times \Sigma,\, G)$ consider the measurable observable $\zeta_g:\Gamma \to \mathbb{R}$
	\begin{equation}
	    \label{zetaG}
		\zeta_g(\omega_0,\omega_1, \xi) :=   h_\xi(g(\omega_0,\omega_1)x_0).
	\end{equation}
	where $h_\xi$ is the horofunction related to $\xi$ through the local minimum map homeomorphism. Measurability of $\zeta_g$ follows from continuity. Notice that the set $\Omega \subset \Gamma^\mathbb{N}$ consisting of sequences $\kappa_n  = (\omega_{n-1}, \omega_{n}, \xi_n)$, where $\xi_n = (g(\omega_0, \omega_1)g(\omega_1, \omega_2)...g(\omega_{n-1},\omega_n))^{-1} \xi_0$ and notice that this is a set of full measure. The sum process in $\Sigma^\mathbb{N}$ is 
	\begin{align*}
		(S_n \zeta)(\omega) &= \sum_{i=0}^{n-1} \zeta(\omega_i, \omega_{i+1}, \xi_i)\\
		&= \sum_{i=0}^{n-1} h_{\xi_i}(g(\omega_i, \omega_{i+1}) x_0)  \\
		& = \sum_{i=0}^{n-1} (g(\omega_0, \omega_1)g(\omega_1, \omega_2)...g(\omega_i,\omega_{i+1}))^{-1} \cdot h_{\xi_0}(g_ix_0)\\
		&= \sum_{i=0}^{n-1} h_{\xi_0}(g^{(i+1)}(\omega)x_0) - h_{\xi_0}(g^{(i)}(\omega)x_0) \\
		& =  h_{\xi_0}(g^{(n)}(\omega)x_0).
	\end{align*}
	These equalities are mostly a consequence of the property $g\cdot h_\xi = h_{g\xi}$ and $(\ref{horoAction})$. In what follows we will prove that provided $g\in C(\Sigma \times \Sigma,\,  G)$ is irreducible with positive drift, then there exists a unique $\overbar{K}_g$-stationary measure which we denote by $\mu_g$. Finally we consider the space of observed Markov systems 
	\begin{equation*}
		\mathcal{M} := \{ (\overbar{K}_g, \mu_g, \pm \zeta_g) \, : \, g\in C(\Sigma \times \Sigma,\,  G), \, \textrm{g is irreducible and } \ell(g)>0 \},
	\end{equation*}
	where $\mu_g$ is the $K_g$-stationary measure, with the metric
	\begin{equation*}
		\textrm{dist}((\overbar{K}_{g_1}, \mu_{g_1}, \zeta_{g_1}), (\overbar{K}_{g_2}, \mu_{g_2}, \zeta_{g_2}) ) := d_\infty(g_1, g_2).
	\end{equation*}
    Due to the metric used, neighbourhoods in $\mathcal{M}$ are naturally identified with neighbourhoods in $C(\Sigma \times \Sigma,\,  G)$.
	Our main goal for the remainder of this section is to prove the following proposition:
	
	\begin{proposition}
		\label{conditionsC}
		The space $\mathcal{M}$ satisfies the $C*)$ conditions.
	\end{proposition}
	
	Notice that the $(\overbar{Q}_g f)(\omega_0, \omega_1, \xi)$ does not depend on the variable $\omega_0$. So we define $\mathcal{H}_\alpha(\Sigma \times \partial X)$ to be the space of functions $f$ in $\mathcal{H}_\alpha(\Gamma)$ that do not depend on $\omega_0$. Notice as well that $\mathcal{H}_\alpha(\Sigma \times \partial X)$ is still a family of Banach algebras satisfying $A*)$ and $B*)$. Our first goal is to prove that this space is invariant under the action of $\overbar{Q}_g$:
	
	\begin{proposition}
		The space $\mathcal{H}_\alpha(\Sigma \times \partial X)$ is invariant by the action of $Q_g$ for $\alpha$ small enough.
	\end{proposition}
	
	The proof of this proposition is based of the Lemmas \ref{contractiveLemma}, \ref{submultiLemma} and \ref{boundedAH}. First, given $g \in C(\Sigma \times \Sigma, \, G)$ and $0< \alpha< 1$ define the average Hölder constant of $g$ as
	\begin{equation*}
		k_\alpha^n(g) := \sup_{\omega_0\in \Sigma, \xi \neq \eta} \mathbb{E}_{\omega_0} \left[ \left( \frac{D_b(g^{-(n)}\xi \, , \, g^{-(n)}\eta )}{D_b(\xi\, , \, \eta)} \right)^ \alpha \right].
	\end{equation*}
	The relevance of $k_\alpha^n(g)$ becomes evident in the following lemma where we relate it with the contracting behaviour of the Markov operator of $g$.
	
	\begin{lemma}[in \cite{duarte2016lyapunov}]
		\label{contractiveLemma}
		Given $g\in S^\infty(\Sigma \times \Sigma,\,  G)$, $f\in \mathcal{H}_\alpha(\Sigma \times \partial X)$ and $n\in \mathbb{N}$,
		\begin{equation*}
			\upsilon_\alpha(\overbar{Q}_g^nf) \leq k_\alpha^n(g) \upsilon_\alpha(f).
		\end{equation*}
	\end{lemma}
	
	\begin{proof}
		Let $f\in \mathcal{H}(\Sigma \times \partial X)$ and $(\omega_0, \xi)\in \Sigma \times \partial X$, recall as well the random variables $e_n:\Omega \to \Sigma$ given by $e_n(\omega)=\omega_n$. Then notice
		\begin{equation*}
			(\overbar{Q}_g^nf)(\omega_0,\xi) = \mathbb{E}_{\omega_0}\left[ f(e_n,g^{-(n)}\xi)\right].
		\end{equation*}
		Hence
		\begin{align*}
			\upsilon_\alpha(\overbar{Q}_g^nf)
			& \leq  \sup_{\omega_0 \in \Sigma, \xi \neq \eta \in \partial X} \frac{\mathbb{E}_{\omega_0} \left|f(e_n, g^{-(n)} \xi) - f(e_n, g^{-(n)} \eta)\right|}{D_b(\xi\, , \, \eta)} \\
			& \leq \upsilon_\alpha(f) \sup_{\omega_0\in \Sigma, \xi \neq \eta} \mathbb{E}_{\omega_0} \left[ \left( \frac{D_b(g^{-(n)}\xi \, , \, g^{-(n)}\eta )}{D_b(\xi\, , \, \eta)} \right)^ \alpha \right] \leq \upsilon_\alpha(f) \, k_\alpha^n(g)
		\end{align*}
	\end{proof}

	\begin{lemma}[in \cite{duarte2016lyapunov}]
		\label{submultiLemma}
		The sequence $(k_\alpha^n(g))$ is sub-multiplicative, that is,
		\begin{equation*}
			k_\alpha^{n+m}(g) \leq k_\alpha^{n}(g)k_\alpha^{m}(g)
		\end{equation*}
	\end{lemma}
	
	\begin{proof}
		See \cite{duarte2016lyapunov}.
	\end{proof}
	
	\begin{lemma}
	    \label{boundedAH}
		Given $g\in S^\infty(\Sigma \times \Sigma,\,  G)$ and $n\in \mathbb{N}$, for every $0<\alpha<\frac{1}{n}$ there exists a constant $C=C(g)$, such that
		\begin{equation*}
			k_\alpha^{n}(g) \leq d_\infty(g).
		\end{equation*} 
	\end{lemma}
	
	\begin{proof}
		Given $\omega_0 \in \Sigma$ and $\xi\neq  \eta$ in $\partial X$, using Proposition \ref{VisualMetric}
		\begin{align*}
			\mathbb{E}_{\omega_0} \left[ \left( \frac{D_b(g^{-(n)}\xi \, , \, g^{-(n)}\eta )}{D_b(\xi\, , \, \eta)} \right)^ \alpha \right] 
			& \leq \mathbb{E}_{\omega_0} \left[ b^{-\frac{\alpha}{2}(h_\xi(g^{(n)}x_0)+h_\eta(g^{(n)}x_0))} \right] \\
			& \leq \mathbb{E}_{\omega_0} \left[ b^{\alpha d(g^{(n)}x_0, x_0)} \right] \\
			& \leq \mathbb{E}_{\omega_0} \left[b^{ d(g x_0, x_0)} \right] \leq d_\infty(g)
		\end{align*}
		taking the supremum in ${\omega_0}$ and $\xi \neq \eta$ we obtain the statement using Lemma \ref{normIneq}.
	\end{proof}
	
	The following Lemma is where the necessity for the hyperbolic multiplicative ergodic theorem appears. The last part of the proof is analogous to that of \cite{bougerol1988theoremes} for the uniformity of the limit for Lyapunov exponents.
	
	\begin{lemma}
		\label{unifConv}
		Let $g\in C(\Sigma \times \Sigma,\,  G)$ be irreducible with positive drift,
		\begin{equation}
		    \label{unifLimit}
			\lim_{n\to \infty} \frac{1}{n}\mathbb{E}_{\omega_0} \left[ h(g^{(n)}x_0)\right] = \ell(g)
		\end{equation}
		uniformly on $(\omega_0, h)\in \Sigma \times X_\infty^h$.
	\end{lemma}
	
	We warn the reader that in the following proof we work with $\Gamma_1 = \Sigma \times \Sigma \times X^h$. We do this as we need compacity. With that in mind we are going to use the Markov Kernel in $\Gamma_1$ analogous to the one used in $\Gamma$, that is
    \begin{equation*}
		\overbar{K}_g(\omega_0, \omega_1, h) := \int_\Sigma \delta_{(\omega_1, \omega_2, g(\omega_1, \omega_2)^{-1} \cdot h)} K(\omega_1, d\omega_2),
	\end{equation*}	
	which in turn gives rise to a Markov operator in the typical fashion. By compactness of $\Gamma_1$, there exists at least one $\overbar{K}_g$-stationary measure $\mu$. In what follows we drop the $g$ in $\overbar{K}_g$ and denote by $\mathbb{P}$ the Kolmogorov extension measure with respect to some $\overbar{K}$-stationary measure in $\Omega = \Gamma_1^\mathbb{N}$.
	
	The strategy of the proof is to first prove that the limit exists for every horofunction $h$ and $\mu$ almost every $\omega$. Then prove it is uniform on $h$ and finally obtain its uniformity on $\omega_0$. With that in mind we will prove four claims, are the Lemma should follow once those are done.
	
	\bigskip
	
	\noindent
	\textbf{Claim 1:} For every $h\in X_\infty^h$, $\lim_{n\to \infty}\frac{1}{n} h(g^{(n)}(\omega)x_0) = \ell(g)$ holds for $\mathbb{P}$ almost every $\omega$. 
	
	\begin{proof}[Proof of Claim 1]
	    Consider the observable $\zeta\colon \Gamma_1 \to \mathbb{R}$ defined by
		\begin{equation*}
		    \zeta(\omega_0,\omega_1, h):= h(g(\omega_0,\omega_1) x_0) 
		\end{equation*} 
		which is clearly continuous. Denote by $\Prob_K(\Gamma_1)$ the space of $K$-stationary probability measures on $\Gamma_1$, which is non-empty by compactness of $\Gamma_1$. Just as before, consider the sum process $S_n\zeta$ generated by $\zeta$ along a $K$-Markov process on $\Gamma_1$ with initial state $(\omega_0,\omega_1, h_0)\in\Gamma$. This sum process can be realized as the process on $\Omega=\Sigma^\mathbb{N}$ defined by 
		\begin{equation*}
		    (S_n\zeta)(\omega):= \sum_{j=0}^{n-1} \zeta(\omega_i, \omega_{i+1}, h_i) = h_0(g^{(n)}(\omega)\, x_0)
		\end{equation*}
		where $h_{i+1}=g(\omega_{i-1},\omega_i)^{-1}\cdot h_i$ for every $i\geq 0$.

        By Furstenberg-Kifer Theorems 1.1 and 1.4 in  \cite{furstenberg1983random}, letting 
        \begin{equation*}
            \beta : = \sup\left\{  \int_\Gamma \zeta\, d \eta \, : \, \eta\in \Prob_K(\Gamma_1) \right\}
        \end{equation*}
        then for $\mathbb{P}$-almost every $\omega\in\Omega$ 
        \begin{equation*}
            \limsup_{n\to\infty}\frac{1}{n}\, h_0(g^{(n)}(\omega) x_0) = \limsup_{n\to\infty}\frac{1}{n}\, (S_n\zeta)(\omega)\leq \beta .
        \end{equation*}
    
        We claim now that $\int_{\Gamma_1} \zeta\, d\eta=\beta$ for every measure $\eta\in\Prob_K(\Gamma_1)$. Then changing  $\zeta$ by $-\zeta$ , the same argument implies that for $\mathbb{P}$-almost every $\omega$ 
        \begin{equation*}
            \lim_{n\to\infty}\frac{1}{n}\, h_0(g^{(n)}(\omega) x_0) = \beta .
        \end{equation*}
        By the Theorem \ref{hmet} and its remark we must have $\beta=\ell(g)$.
	\end{proof}
	
	\noindent
    \textbf{Claim 2}: $\int_{\Gamma_1} \zeta\, d\eta=\beta$ for every measure $\eta\in\Prob_K(\Gamma_1)$.
    
    \begin{proof}[Proof of Claim 2]
        If the claim were false there would be an ergodic measure $\eta\in\Prob_K(\Gamma_1)$ such that $\int_{\Gamma_1} \zeta\, d\eta=\beta_1 <\beta$. Consider the map  
        \begin{align*}
            F:\Omega \times X^h &\to \Omega \times X^h \\
            (\omega, h) &\mapsto (\sigma\omega, g(\omega_0, \omega_1)^{-1}\cdot h)
        \end{align*}
        which preserves the ergodic measure $\mathbb{P}\times\eta$. The observable $\zeta$ can be extended to $\bar \zeta\colon\Omega\times X^h\to\mathbb{R}$, $\overbar\zeta(\omega, h)=\zeta(\omega_0,\omega_1,h)$. Moreover, with this notation, $(S_n\zeta)(\omega) =\sum_{j=0}^{n-1} \overbar\zeta(F^j(\omega,h_0)) $ is a Birkhoff sum. By Birkhoff's ergodic theorem, for $\eta$-almost every $h_0\in X^h$ and $\mathbb{P}$-almost every $\omega\in\Omega$,
        \begin{equation*}
            \lim_{n\to\infty} \frac{1}{n} h_0(g^{(n)}(\omega) x_0) = \lim_{n\to\infty} \frac{1}{n} \sum_{j=0}^{n-1} \bar\zeta(F^j(\omega,h_0)) = \beta_1
        \end{equation*}
        which together with Theorem \ref{hmet} implies that $\beta_1=-\ell(g)$ and $h_0\in X_-^h(\omega)$. Next consider the family of sets
        \begin{equation*}
            S_{\omega_0}:= \left\{ h\in X^h \,\colon \; \mathbb{P}_{\omega_0}\{ \omega\in\Omega \, \colon h \in X_-^h(\omega)\}=1\, \right\} .
        \end{equation*}
        The previous argument shows that $S_{\omega_0}\neq \emptyset$ for $\mathbb{P}$-almost every $\omega\in\Omega$.  Again by the remark to Theorem \ref{hmet}
        the set $S_{\omega_0}$ must be a single horofunction
        $S_{\omega_0}=\{s(\omega_0)\}$ and the function
        $s\colon \Sigma\to X_h$ is measurable.
        The invariance of $X_-^h$ in Theorem \ref{hmet} now implies that $g(\omega_0,\omega_1)\cdot s(\omega_0)=s(\omega_1)$, which proves that $g$ is not irreducible. This contradiction implies that the claim is true.
    \end{proof}

	\noindent
	\textbf{Claim 3}: The convergence is uniform $h$.
		
	\begin{proof}[Proof of Claim 3]
	    Let us start by proving the uniformity in $h$, arguing by absurd, suppose there is a sequence of horofunctions $(h_n)\subset X_\infty ^h$ converging to some $h$ in $X_\infty^h$ and $\varepsilon>0$ such that
		\begin{equation*}
			\lim_{n\to \infty} \frac{1}{n}\mathbb{E}_{\omega_0} \left[ h_n(g^{(n)}x_0)\right] < \ell(g)- \varepsilon
		\end{equation*}
		Due to the compactness of $X^h$ we can assume that $h_n$ converges. Take $(y_m^n)_m \in X$ and $\xi_n \in \partial X$ two families of sequences such that $h_{y_m^n} \to h_n=:h_{\xi_n}$ and $y_m^n \to \xi_n$ as $m\to \infty$. Then
		\begin{align*}
			\lim_{n\to \infty} h_n(g^{(n)}x_0) - d(g^{(n)}x_0\, , \, x_0)  & = \lim_{n\to \infty} \lim_{m \to \infty} h_{y_m^n}(g^{(n)}x_0) - d(g^{(n)}x_0\, , \, x_0)\\
			& = \lim_{n\to \infty} \lim_{m \to \infty} d(y_m^n\, , \, g^{(n)}x_0) - d(y_m^n\, , \, x_0) - d(g^{(n)}x_0\, , \, x_0) \\
			& = \lim_{n\to \infty} \lim_{m \to \infty} - 2\langle y_m^n, g^{(n)}x_0 \rangle_{x_0} \\
			& = \lim_{n \to \infty} - 2\langle \xi_n, g^{(n)}x_0 \rangle_{x_0},
		\end{align*}
		where the last equality is a consequence of the continuity of the Gromov product in strongly hyperbolic spaces.
		Notice that the quantity $\langle \xi_n, g^{(n)}x_0 \rangle_{x_0}$ goes to infinity if and only if both $\xi_n$ and $g^{(n)}x_0$ converge to the same point in $\partial X$. If this were the case, by Proposition 4 in \cite{sampaio2021regularity}, $\lim_{n \to \infty}h(g^{(n)}x_0) = -\infty$, hence $h\in X_-^h(\omega)$. Therefore $\langle \xi_n, g^{(n)}x_0 \rangle_{x_0}$ must $\mathbb{P}_{\omega_0}$ almost surely be finite as otherwise $h\in S_{\omega_0}=\emptyset$. Using dominated convergence theorem again,
		\begin{align*}
			\lim_{n\to \infty} \frac{1}{n}\mathbb{E}_{\omega_0} \left[ h_n(g^{(n)}x_0)\right] & = \lim_{n\to \infty} \frac{1}{n}\mathbb{E}_{\omega_0} \left[ d(g^{(n)}x_0\, , \, x_0)\right] + \lim_{n\to \infty} \frac{1}{n}\mathbb{E}_{\omega_0} \left[ h_n(g^{(n)}x_0)-d(g^{(n)}x_0\, , \, x_0)\right] \\
			& = \ell(g) + 0 = \ell(g),
		\end{align*}
		which yields the claim.
	\end{proof}
	
	\noindent	
	\textbf{Claim 4}: The convergence is uniform in $\omega_0$.
	
	\begin{proof}[Proof of Claim 4]
	    Consider now, for $\omega_0\in \Sigma$
		\begin{equation*}
			q_n(\omega_0)=\sup\left\{\left|\frac{1}{n}\mathbb{E}_{\omega_0} \left[ h(g^{(n)}x_0)\right] - \ell(g)\right|\, : \, h \in X^h  \right\},
		\end{equation*}
		and notice the uniform bound $|q_n(\omega_0)| \leq \log_b(d_\infty(g)) +\ell(g)$. Due to the uniform limit in $h$ proven above, using dominated convergence theorem
		\begin{equation*}
			\lim_{n\to \infty} \int_\Sigma p_n(\sigma) d\mu(\sigma)=0.
		\end{equation*}
		Let $\varepsilon>0$. Consider $n>p$ to be specified later and take  $a=a(p):=\sup_{\omega_0}q_p(\omega_0)$
		\begin{align*}
			\left|\frac{1}{n}\mathbb{E}_{\omega_0} \left[ h(g^{(n)}x_0)\right] - \ell(g)\right| & \leq \left|\frac{1}{n}\mathbb{E}_{\omega_0} \left[ g^{-(p)}\cdot h(g^{(n-p)}x_0) + h(g^{(p)}x_0)\right] - \ell(g)\right| \\
			& \leq  \left|\frac{1}{n}\mathbb{E}_{\omega_0} \left[ g^{-(p)}h(g^{(n-p)}x_0)\right] - \ell(g)\right| + \frac{1}{n}\mathbb{E}_{\omega_0} \left[ h(g^{(p)}x_0)\right] \\
			& \leq \left(\frac{n-p}{n}\right)\left| \mathbb{E}_{\omega_0}\left[ \frac{1}{n-p}\mathbb{E}_{\omega_p} \left[ g^{-(p)}h(g^{(n-p)}x_0)\right] - \ell(g)\right]\right| + \frac{p}{n}(\ell(g)+a),
		\end{align*}
		from which
		\begin{equation*}
			q_n(\omega_0) \leq (Q^p q_{n-p})(\omega_0) + \frac{p}{n}(\ell(g)+a).
		\end{equation*}
		Now, taking $p$ and $n$ large enough, one has the following inequalities
		\begin{equation*}
		    \frac{p}{n}(\ell(g)-a)/n < \varepsilon/3,
		\end{equation*}
		as well as
		\begin{equation*}
		     \int_\Sigma q_{n-p}(\sigma)d\mu(\sigma)< \varepsilon/3,
		\end{equation*}
		moreover, by the strongly mixing condition
		\begin{equation*}
			\sup_{\omega_0\in \Sigma} \left|(Q^p q_{n-p})(\omega_0) - \int_\Sigma q_{n-p}(\sigma)d\mu(\sigma)\right| \leq \varepsilon/3,
		\end{equation*}
		provided $p$ is large enough and taking $n$ large enough. Hence
		\begin{equation*}
		    q_n(\omega_0) \leq \int_\Sigma q_{n-p}(\sigma)d\mu(\sigma) + 2\varepsilon/3 < \varepsilon.
		\end{equation*}
	\end{proof}
		
	\begin{proof}[Proof of Lemma \ref{unifConv}]
	    Notice that the uniform convergence on $\omega_0$ follows from the uniform convergence on $h$, hence we have joint uniform convergence on both.
	\end{proof}
	
	In the following proposition we will use the relation, which is an immediate consequence of Proposition \ref{VisualMetric},
	\begin{equation*}
		k_\alpha^n(g) \leq \sup_{\omega \in \Sigma, \, \xi \in \partial X} \mathbb{E_\omega} \left[ b^{- \alpha h_\xi(g^{(n)}x_0)} \right].
	\end{equation*}

	\begin{proposition}
		Given $g_1\in C(\Sigma \times \Sigma, G)$ irreducible with positive drift, there exists a neighbourhood $V$ of $g_1$ in $C(\Sigma \times \Sigma, G)$ and constants $n_0 \in \mathbb{N}$, $0<\alpha_1 < \alpha_0/2 < \alpha_0$, $C=C(g_1)>0$,  and $0\leq \sigma <1$ such that 
		\begin{equation*}
			k_\alpha^n(g_2) \leq C\sigma^n,
		\end{equation*}
		for all $g_2 \in V$, $n>n_0$, $\alpha \in [\alpha_0, \alpha_1]$ and $f\in \mathcal{H}_\alpha(\Sigma \times \partial X)$.
	\end{proposition}
	
	\begin{proof}

		By Lemma \ref{unifConv}, 
		\begin{equation*}
			\lim_{n \to \infty} \sup_{ \xi \in \partial X }
			\left| \mathbb{E}_{\omega_0} \left[  h_\xi(g^{(n)}x_0)\right] - \ell(g_1)  \right|= 0.
		\end{equation*}
		In particular, there exists $n_0\in \mathbb{N}$ such that $\mathbb{E}_{\omega_0} \left[ h_\xi(g^{(n_0)}x_0)\right] \geq \frac{1}{\log b}>0$ for every $\xi \in \partial X$.
		
		Let $r>0$ to be specified later and consider in $C(\Sigma \times \Sigma, G)$ the neighbourhood of $g_1$ given by the ball
		\begin{equation*}
			V = B_r(g_1) := \{ g_2 \in C(\Sigma \times \Sigma, G) \, : \, d_\infty(g_1\, , \, g_2) < r \}.
		\end{equation*}
		Let $g_2\in B_r(g_1)$, $\omega_0\in \Sigma$, $\xi$ in $\partial X$.   
		Use the inequality 
		\begin{equation*}
			b^x < 1 + \log(b)x + \log(b)^2\frac{x^2}{2} b^{|x|},
		\end{equation*}
		to obtain,
		\begin{align*}
			\mathbb{E}_{\omega_0} \bigg[ b^{-\alpha h_\xi(g_1^{(n_0)}x_0)} \bigg] 
			& \leq   1 - \alpha\log(b) \mathbb{E}_{\omega_0}\bigg[ h_\xi(g_1^{(n_o)}x_0)\bigg] \\
			& \hspace{1cm}+ \log(b)^2\frac{\alpha^2}{2} \mathbb{E}_{\omega_0}\left[ (h_\xi(g_1^{(n_0)}x_0))^2 b^{|\alpha h_\xi(g_1^{(n_0)}x_0)|}\right] \\
			& \leq 1 - \alpha + \alpha^2 b^\alpha \left( \frac{\log(b)^2}{2}n_0^2 \log_b(C)^2 C^{n_0\log b}\right),
		\end{align*}
		where $C$ is a constant depending on $g_1$. Hence there exists $\alpha$ small enough so that the right-hand side becomes smaller than $1$, which implies the existence of constants $\alpha_0$ and $\alpha_1 < \alpha_0/2$ such that $k_\alpha^{n_0}(g_1) < \rho <1$.

		To extend this control to nearby cocycles let us introduce the following continuity type relation, using the mean value theorem and the argument around finite scale continuity in Proposition \ref{FiniteScaleContinuity}
		\begin{align*}
			\mathbb{E}_{\omega_0} \left|b^{-\alpha h_\xi(g_1^{(n_0)}x_0)} - b^{-\alpha h_\xi(g_2^{(n_0)}x_0)} \right| & 
			\leq (\log b) \max_{i=1,2} b^{d(g_i^{(n_0)} x_0, x_0)}      \left|h_\xi(g_1^{(n_0)}x_0) - h_\xi(g_2^{(n_0)}x_0)\right| \\
			& \leq C^{n_0} d(g_1^{(n_0)}x_0,g_2^{(n_0)}x_0)  \\
			& \leq n_0C^{2n_0} d_\infty(g_1, g_2)
		\end{align*}
		We can now choose $r$ small enough to ensure there exists $\rho^* \in (\rho,1)$ such that
		\begin{align*}
			\mathbb{E}_{\omega_0} \left|b^{-\alpha h_\xi(g_1^{(n_0)}x_0)} - b^{-\alpha h_\xi(g_2^{(n_0)}x_0)} \right| \leq \rho^* - \rho.
		\end{align*}
		Hence
		\begin{align*}
			\mathbb{E}_{\omega_0}\left[b^{-\alpha h_\xi(g_2^{(n_0)}x_0)} \right] & \leq  \mathbb{E}_{\omega_0}\left[b^{-\alpha h_\xi(g_1^{(n_0)}x_0)} \right] + \left|\mathbb{E}_{\omega_0}\left[b^{-\alpha h_\xi(g_2^{(n_0)}x_0)} \right] - \mathbb{E}_{\omega_0}\left[b^{-\alpha h_\xi(g_1^{(n_0)}x_0)} \right] \right| \\
			& \leq  \mathbb{E}_{\omega_0}\left[ b^{-\alpha h_\xi(g_1^{(n_0)}x_0)}  \right] + \mathbb{E}_{\omega_0} \left|b^{-\alpha h_\xi(g_1^{(n_0)}x_0)} - b^{-\alpha h_\xi(g_2^{(n_0)}x_0)} \right|\\
		    &  \leq \rho + (\rho^*-\rho) = \rho^* <1
		\end{align*}
		
		Due to the submultiplicativity, picking $\sigma = (\rho^*)^{\frac{1}{n_0}}$, for every $n \in \mathbb{N}$ there exists a constant $C>0$ such that
		\begin{equation*}
			k_\alpha^n(g_2)\leq \mathbb{E}_{\omega_0} \left[b^{-\alpha h_\xi(g_2^{(n)}x_0)}\right] < C \sigma^n,
		\end{equation*}
		which completes the proof.
	\end{proof}
	
	The previous proposition now allows us to obtain the existence and uniqueness of the $K_g$ stationary measures $\mu_g$ in a neighbourhood of $g$ irreducible with positive drift.
	
	\begin{proposition}
	    Let $g\in S^\infty(\Sigma \times \Sigma, G)$ have positive drift. If for some $n\in \mathbb{N}$ and $\alpha<1$ 
	    \begin{equation*}
	        k_\alpha^n(g)^{1/n} < 1,
	    \end{equation*}
	    then there exists a unique $K_g$-stationary measure.
	\end{proposition}
	
	\begin{proof}
	    The proof is mostly taken from \cite{sampaio2021regularity}. The seminorms $\upsilon_\alpha$ are norms in the space $\mathcal{H}_\alpha(\Gamma) / \mathbb{C}\mathbf{1}$. Since $\overbar{Q}_g^n \mathbf{1}=\mathbf{1}$, by hypothesis, $\overbar{Q}_g^n$ acts in $\mathcal{H}_\alpha(\Gamma) / \mathbb{C}\mathbf{1}$ as a contraction. Using spectral theory (see chapter IX in \cite{riesz2012functional} for example), there exists and invariant space $H_0$, isomorphic to $\mathcal{H}_\alpha(\Gamma) / \mathbb{C}\mathbf{1}$, such that $\mathcal{H}_\alpha(\Gamma) = H_0 \oplus \mathbb{C}\mathbf{1}$. Given $f\in \mathcal{H}_\alpha(\Gamma)$ we may write it as $c\mathbf{1} + h$ where $c\in \mathbb{C}$ and $h\in H_0$. With that in mind, define
		\begin{align*}
			\Lambda : \mathcal{H}_\alpha(\Gamma) & \to \mathbb{C} \\
			c\mathbf{1}+h \mapsto c.
		\end{align*}
		Now notice that $\overbar{Q}_g$ is a positive operator, therefore so is $\Lambda$ as
		\begin{equation*}
			c\mathbf{1}=\lim_{n\to \infty}\left(c\mathbf{1} + \overbar{Q}_g^n(h)\right)= \lim_{n \to \infty} \overbar{Q}_g^n(f) \geq 0,
		\end{equation*}
		provided $f\geq 0$. Hence $c=\Lambda(f)\geq 0$. Positivity also implies continuity with respect to the uniform norm as
		\begin{equation*}
			|\Lambda(\varphi)| \leq |\Lambda(||\varphi||_\infty \mathbf{1})|| = ||\varphi||_\infty.
		\end{equation*}
		
		Now since $\Gamma$ is a metric space, the set of bounded Lipschitz functions in $\partial X$ is dense in the space of bounded uniformly continuous functions $C_b(\Gamma)$. With effect, given $f\in C_b(\Gamma)$ one can take the functions
		\begin{equation*}
			f_n(\xi) = \inf_{\eta\in \Gamma} \{f(\eta)-nD_b(\xi, \eta)\},
		\end{equation*}
		which are all bounded Lipschitz and uniformly converge to $f$. Since the space is bounded, the set of Lipschitz functions is contained in the space of Hölder functions, so $\mathcal{H}_\alpha(\Gamma)$ is dense in $C_b(\Gamma)$. Hence, $\Lambda$ extends to a positive linear continuous functional $\hat{\Lambda}:C_b(\Gamma) \mapsto \mathbb{C}$. 
		
		Riesz-Kakutani-Markov for non-compact spaces (Theorem 1.3 in \cite{sentilles1972bounded}) applies, so there exists a measure $\nu\in \Prob(\Gamma)$ such that $\hat{\Lambda}(f)=\int_{\Gamma} f d\nu$ for every $f\in C_b(\Gamma)$. Finally, writing $f$ once again as $c\mathbf{1} + h$ yields 
		\begin{equation*}
			\int_{\Gamma} \overbar{Q}_g f d\nu = \hat{\Lambda}(\overbar{Q}_g f) = c = \hat{\Lambda}(f)=\int_{\Gamma} f d\nu.
		\end{equation*}
		By yet another density argument, this holds for all $f\in L^1(\Gamma)$, therefore $\nu$ is $K_g$-stationary. This density of $C_b(\partial X)$ in $L^1(\Gamma)$ also justifies the uniqueness of the measure satisfying $\hat{\Lambda}(f)=\int_{\Gamma} f d\nu$.
	\end{proof}
	
	Henceforth $\mathcal{M}$ is welll defined and condition C1) is immediate. We now focus the remaining conditions.
	
	\begin{proposition}
		\label{condiditionC2}
		Given $g_1 \in C(\Sigma \times \Sigma, G)$ such that $(\overbar{K}_{g_1}, \mu_{g_1}, \zeta_{g_1}) \in \mathcal{M}$, there exist a
		neighborhood $V$ of $g_1$ in $C(\Sigma \times \Sigma, G)$,
		constants $0< \alpha_1 < \alpha_0/2 < \alpha_0<1$, $C > 0$ and $0 < \sigma < 1$ such that for all $g_2 \in V$, and
		$f \in \mathcal{H}_\alpha(\Sigma \times \partial X)$,
		\begin{equation*}
			\left|\left|\overbar{Q}_{g_2}^nf - \int_\Sigma f(\omega) d\mu_{g_2}(\omega)\right|\right|_\alpha \leq C\sigma^n ||f||_\alpha.
		\end{equation*}
	\end{proposition}
	
	\begin{proof}
		Take the neighbourhood $V$ from the previous proposition, given $g_2\in V$ and any $K_{g_2}$ stationary measure $\mu_{g_2}$,
		\begin{equation}
		    \label{exponentialConv}
			\upsilon_\alpha \left(\overbar{Q}_{g_2}^nf - \int_\Sigma f(\omega)d\mu_{g_2}(\omega)\right) = \upsilon_\alpha(\overbar{Q}_{g_2}^nf) \leq \upsilon_\alpha(f)k_\alpha^n(g_2) \leq C\sigma^n ||f||_\alpha.
		\end{equation}
		So it remains to prove
		\begin{equation}
		    \label{exponentialGoal}
			\left|\left|\overbar{Q}_{g_2}^nf - \int_\Sigma f(\omega) d\mu_{g_2}(\omega)\right|\right|_\infty \leq 2\left|\left|\overbar{Q}_{g_2}^nf\right|\right|_\infty \leq C\sigma^n ||f||_\alpha,
		\end{equation}
	    for possibly some other $C<\infty$ and $0<\sigma<1$.
	    
        For this purpose, consider as well as the operator
        $ Q : L^\infty(\Sigma) \to L^\infty(\Sigma)$
        \begin{equation*}
            (Q  f)( \omega_1):=  \int_{\Sigma} f( \omega_2)\, dK_{\omega_1}(\omega_2). 
        \end{equation*}
         
        There is a natural projection $\pi : \Sigma\times \partial X\to \Sigma$ which induces a bounded linear embedding $\pi^\ast : L^\infty(\Sigma) \to \mathcal{H}_\alpha(\Sigma\times \partial X)$, $\pi^\ast f := f\circ \pi$. Notice that the range of this embedding is the subspace
        \begin{equation*}
            \pi^\ast L^\infty(\Sigma)=\left\{ f\in \mathcal{H}_\alpha(\Sigma\times X^h)\, : \, v_\alpha(f)=0 \, \right\}
        \end{equation*}
        and the following diagram commutes for every $n\in\mathbb{N}$
        \begin{equation*}
            \begin{CD}
            	L^\infty(\Sigma)  @>Q^n >> L^\infty(\Sigma) \\
            	@V\pi^\ast VV @VV \pi^\ast V\\
            	\mathcal{H}_\alpha(\Sigma\times \partial X)  @>>  \overbar{Q}^n > \mathcal{H}_\alpha(\Sigma\times \partial X) .
            \end{CD}
        \end{equation*}
        
        Given $f\in\mathcal{H}_\alpha(\Sigma\times X^h)$, by (\ref{exponentialConv}) the iterates $\overbar{Q}^n f$ converge exponentially fast to the closed  subspace $L^\infty(\Sigma) \equiv \pi^\ast L^\infty(\Sigma) \subseteq \mathcal{H}_\alpha(\Sigma\times X^h)$. On the other hand by assumption $Q$ is strongly mixing on $L^\infty(\Sigma)$. Combining these two properties and the fact that Markov operators do not expand we get (\ref{exponentialGoal}).
	\end{proof}
	
	The Laplace-Markov operator $Q_{g,z}$ of the observed Markov system $(K_g,\mu_g, \zeta_g)$ is given by 
	\begin{equation*}
		(Q_{g,z}f)(\omega_0, \omega_1, \xi) = \int_{\Sigma} f(\omega_1, \omega_2, g(\omega_1, \omega_2)^{-1} \xi) b^{zh_\xi(g(\omega_1, \omega_2) x_0)}\, K(\omega_1, d\omega_2). 
	\end{equation*}

	\begin{lemma}
		\label{conditionC4}
		Given $g_1, g_2 \in S^\infty(\Sigma \times \Sigma, G)$ and $b>0$, there is a constant $C_2 > 0$ such that for all $f \in $$\mathcal{H}_\alpha(\Sigma \times \partial X)$ and all $z \in \mathbb{C}$ such that $\textrm{Re}\, z \leq c$,
		\begin{equation*}
			||Q_{g_1, z}f - Q_{g_2, z}f||_\infty  \leq C_2 d_\infty(g_1,g_2)^\alpha||f||_\alpha.
		\end{equation*}
		Moreover, $C_2$ is bounded on a neighborhood of $g_1$.
	\end{lemma}
	
	\begin{proof}
		Let $\xi \in \partial X$. Start by noticing that writting $z=x + yi$ with $x \leq c$
		\begin{align*}
			|b^{zh_\xi(g_1 x_0)}-b^{zh_\xi(g_2 x_0)}| & \leq \max_{i=1,2} b^{c\, d(g_i x_0, x_0)}      \left|c\,h_\xi(g_1x_0) -  c\,  h_\xi(g_2x_0)\right|\\
			& \leq c\, d_\infty(g_1,g_2) \max_{i=1,2} d_\infty(g_i)^{c}.
		\end{align*}
		Hence
		\begin{align*}
			|Q_{g_1, z}f & - Q_{g_2, z}f| \leq \mathbb{E}_{\omega_0} \left[|b^{zh_\xi(g_1x_0)}f(e_1, g_1^{-1} \xi) - b^{zh_\xi(g_2x_0)}f(e_1, g_2^{-1} \xi)|\right]\\
			& \leq ||f||_\infty \mathbb{E}_{\omega_0} \left[|b^{zh_\xi(g_1x_0)}-b^{zh_\xi(g_2x_0)}| \right] + \max_{i=1,2} d_\infty(g_i)^{c} \mathbb{E}_{\omega_0} \left[|f(e_1, g_1^{-1}\xi) - f(e_1, g_2^{-1} \xi)|\right]\\
			& \leq c\, d_\infty(g_1,g_2) \max_{i=1,2} d_\infty(g_i)^{c} ||f||_\infty + \max_{i=1,2} d_\infty(g_i)^{c} \upsilon_\alpha(f) \mathbb{E}_{\omega_0} \left[D_b(g_1^{-1}\xi, g_2^{-1}\xi)^\alpha\right] \\
			& \leq C_2 ||f||_\alpha d_\infty(g_1, g_2)^\alpha.
		\end{align*}
		where $C_2 = \max \left\{ c\, \max_{i=1,2} d_\infty(g_i)^{c}, \max_{i=1,2} d_\infty(g_i)^{c} \right\}$ which are bounded in a neighbourhood $g_1$. The last inequality is a consequence of $d_\infty(g_1, g_2) < d_\infty(g_1, g_2)^\alpha< 1$ and $D_b(g_1^{-1}\xi, g_2^{-1}\xi) \leq d_\infty(g_1, g_2)$.
	\end{proof}
	
	\begin{proof}[Proof of Proposition \ref{conditionsC}]
		Point $C1)$ is obvious. For $C3)$ recall from (\ref{zetaG}) that $||b^{\zeta_g}||_\infty = ||b^{h(gx_0)}||_\infty \leq d_\infty(g)$ which is finite, hence $\zeta_g \in \mathcal{H}_\alpha(\Sigma \times \partial X)$. Therefore $Q_{g,z\zeta}$ acts on $\mathcal{H}_\alpha(\Sigma \times \partial X)$ as the latter is a Banach algebra. Point $C2)$ is a consequence of Proposition \ref{condiditionC2} while $C4)$ follows from the previous Lemma.
	\end{proof}
	
	With this, we can now obtain the large deviations.
	
	\begin{proof}[Proof of Theorem \ref{LDT}]
		Using Theorem \ref{LDTT}, there exists $V$ a neighbourhood of $g \in S^\infty(K)$ and constants $\varepsilon_0, C, k>0$ such that for every $g_2\in V$, $0<\varepsilon<\varepsilon_0$, $h \in X_\infty^h$ and $n\in \mathbb{N}$
		\begin{equation*}
			\mathbb{P}_\mu\left[ \left| \frac{1}{n} h(g_2^{(n)}x_0)- \ell(g_2) \right| > \varepsilon \right] \leq C b^{-k\varepsilon^2n}.
		\end{equation*} 
		Using Lemma 5 in  \cite{sampaio2021regularity}, one obtains that there exists an horofunction $h\in X_\infty^h$ such that
		\begin{equation*}
			h(g_2^{(n)}x_0) \leq d(g_2^{(n)}x_0,x_0) \leq h(g_2^{(n)}x_0) + K(\delta).
		\end{equation*}
		where $K(\delta)$ is a constant depending on $\delta$. Using this inequality we obtain the large deviations with a possible loss in the constant $C$.
	\end{proof}
	
	\section*{Acknowledgements}

    The author was supported by the University of Lisbon, under the PhD scholarship program: BD2018. I woud also like to thank professor Pedro Duarte for all the comments and remarks.
	
	\bibliographystyle{plain}
	\bibliography{biblio}
	
	Departamento de Matemática, Faculdade de Ciências, Universidade de Lisboa, Portugal
	
	Email: lmsampaio@fc.ul.pt

\end{document}